\documentclass{article}

\usepackage{authblk}
\usepackage[T1]{fontenc}

\usepackage[utf8]{inputenc}
\usepackage{multicol}
\usepackage{yfonts}
\usepackage{hyperref,amsmath,amssymb,amsthm,units,stmaryrd}
\usepackage{upgreek,xcolor,bm,tikz,stackrel}
\usepackage[inline]{enumitem}
\usepackage{graphicx}
\usepackage[all]{xy}
\usepackage{cancel}
\usepackage[numbers]{natbib}
\newcommand{\remove}[2]{#1 \setminus #2}
\newcommand{\ignore}[1]{}
\newcommand{\peqT}{\peq_{\type{}}}
\newcommand{\sqsubT}{\sqsubseteq_{\type{}}}
\newcommand{\ST}{\mathrel S_{\type{}}}
\newcommand{\AxCons}[2]{{\rm CD}(#1,#2)}
\newcommand{\AxBInd}[2]{{\rm BI}(#1,#2)}
\newcommand{\logbasic}{{\bf ITL}^0}
\newcommand{\loghomeo}{{\bf ITL}^{\bf FS}}
\newcommand{\logexp}{{\bf ITL}^{\bf CD}}
\newcommand{\logpers}{{\bf ITL}^{1}}

\newcommand{\itle}{{\bf ITL^e}}
\newcommand{\itlp}{{\bf ITL^p}}
\newcommand{\tnext}{\circ}
\newcommand{\ubox}{\Box}
\newcommand{\diam}{\Diamond}
\newcommand{\val}[1]{\lb #1 \rb}

\newcommand{\type}[1]{\mathbb T_{ #1 }}
\newcommand{\peq}{\preccurlyeq}
\newcommand{\seq}{\succcurlyeq}

\newcommand{\acc}{\peq}

\def\lb{\left\llbracket}
\def\rb{\right\rrbracket}
\def\<{\left (}

\def\>{\right )}
\def\({\left (}
\def\){\right )}

\def\cbra{\left \{}
\def\cket{\right \}}

\def\eqdef{\stackrel{\rm def}{=}}

\newtheorem{theorem}{Theorem}
\newtheorem{lemma}{Lemma}
\newtheorem{example}{Example}
\newtheorem{definition}{Definition}
\newtheorem{proposition}{Proposition}
\newtheorem{corollary}{Corollary}
\newtheorem{question}{Question}
\newtheorem{remark}{Remark}

\title{Axiomatic systems and topological semantics for intuitionistic temporal logic}
\author[1]{Joseph Boudou \footnote{\href{mailto:joseph.boudou@irit.fr}{\tt joseph.boudou@irit.fr}}}
\author[2]{Mart\'{\i}n Di\'eguez \footnote{\href{mailto:martin.dieguez@enib.fr}{\tt martin.dieguez@enib.fr}}}
\author[3]{David Fern\'andez-Duque \footnote{\href{mailto:david.fernandezduque@ugent.be}{\tt david.fernandezduque@ugent.be}}}
\author[1]{Fabi\'an Romero \footnote{\href{mailto:Fabian.Romero@irit.fr}{\tt Fabian.Romero@irit.fr}}}
\affil[1]{IRIT, Toulouse University. Toulouse, France}
\affil[2]{CERV, ENIB,LAB-STICC. Brest, France}
\affil[3]{Department of Mathematics, Ghent University. Ghent, Belgium}

\begin{document}
\maketitle

\begin{abstract}
We propose four axiomatic systems for intuitionistic linear temporal logic and show that each of these systems is sound for a class of structures based either on Kripke frames or on dynamic topological systems. Our topological semantics features a new interpretation for the `henceforth' modality that is a natural intuitionistic variant of the classical one. Using the soundness results, we show that the four logics obtained from the axiomatic systems are distinct. Finally, we show that when the language is restricted to the `henceforth'-free fragment, the set of valid formulas for the relational and topological semantics coincide.
\end{abstract}

\section{Introduction}

Intuitionistic logic enjoys a myriad of interpretations based on computation, information or topology, making it a natural framework to reason about dynamic processes in which these phenomena play a crucial role.
In the areas of {nonmonotonic reasoning,} {knowledge representation (KR),} and {artificial intelligence,} intuitionistic and intermediate logics have played an important role within the successful {answer set programming (ASP)}~\cite{Brewka11} paradigm for practical KR. Great part of its success is due to the impressive advances in implementation of efficient solvers~\cite{LPF+06,gekakasc14b} and its use in a wide range of domains such as {computational biology}~\cite{GebserSTV11}, spatial reasoning~\cite{WBS15}, or configuration~\cite{gekasc11c}.

Central to this paradigm is {equilibrium logic}~\cite{Pearce96}, which characterises the ASP semantics in terms of the intermediate logic of {here and there}~\cite{Hey30} plus a minimisation criterion. Such a definition has led to several extensions of modal ASP~\cite{CP07,CerroHS15} that are supported by intuitionistic-based modal logics like {temporal here and there}~\cite{BalbianiDieguezJelia} and are crucial when characterising the theorem of {strong equivalence}~\cite{LPV01,JH03,CabalarD14}.

There are also several potential applications for intuitionistic temporal logics that are unrelated to ASP.
Davies~\cite{Davies96} has suggested an extension of the Curry-Howard isomorphism \cite{DeGroote1995} to partially evaluated programs by adding a next-time operator $\tnext$.
Maier~\cite{Maier2004Heyting} observed that an intuitionistic temporal logic with infinitary operators including a henceforth operator $\ubox$ could be used for reasoning about safety and liveness conditions in possibly-terminating reactive systems.
Fern\'andez-Duque~\cite{FernandezITLc} has suggested that a logic with `eventually' can be used to provide a decidable framework in which to reason about topological dynamics. It is thus surprising that the computational and proof-theoretic properties of these logics are far from being well-understood.

\paragraph{State-of-the-art.}
There have, however, been some notable efforts in this direction. Kojima and Igarashi \cite{KojimaNext} endowed Davies' logic with Kripke semantics and provided a complete deductive system.
Bounded-time versions of logics with henceforth were later studied by Kamide and Wansing \cite{KamideBounded}.
Both use semantics based on Simpson's bi-relational models for intuitionistic modal logic \cite{Simpson94}.
Since then, Balbiani and the authors have shown that temporal here-and-there is decidable and enjoys a natural axiomatization \cite{BalbianiDieguezJelia}.
They have identified two natural, semantically-defined intuitionistic temporal logics, $\itle$ and $\itlp$, studied bisimulations for these logics \cite{IMLA}, and shown $\itle$ to be decidable \cite{BoudouCSL}. However, the decision procedure does not provide a natural axiomatization, and moreover the decidability of $\itlp$ remains open, despite the latter logic being attractive due to it validating the familiar Fischer Servi axioms \cite{FS84}.

Topological semantics for intuitionistic modal and tense logics have also been studied by Davoren et al.~\cite{DavorenIntuitionistic,Davoren2009}, and Kremer suggested an intuitionistic variant of $\bf LTL$ \cite{KremerIntuitionistic} similar to {dynamic topological logic} ($\bf DTL$) \cite{arte,kmints}. $\bf DTL$ is a tri-modal system which gained interest due to its potential applications to automated theorem proving for topological dynamics, but was later shown to be undecidable \cite{konev}. On the other hand, the decidability of Kremer's intuitionistic temporal logic remains open, but Fern\'andez-Duque has shown that a logic with `eventually' $\diam$ instead of $\ubox$ is decidable \cite{FernandezITLc}. Both intuitionistic temporal logics can be seen as sublogics of $\bf DTL$ via the G\"odel-Tarski translation \cite{tarski}.

\paragraph{Our contribution.}
The above decidability results for intuitionistic temporal logics are based on semantical methods. The primary goal of this paper is to lay the groundwork for an axiomatic treatment of intuitionistic linear temporal logics.
We will introduce a `minimal' intuitionistic temporal logic, $\logbasic$, defined by adding standard axioms of $\bf LTL$ to intuitionistic propositional logic. We also consider additional Fischer Servi axioms and a `constant domain' axiom $\ubox (p\vee q) \to \ubox p \vee \diam q$.
Combining these, we obtain four intuitionistic temporal logics. As we will see, each of these logics is sound for a class of structures; the two logics with the constant domain axiom are sound for the class of dynamic posets, and the Fischer Servi axioms correspond to backwards-confluence of the transition function.

The constant domain axiom is not derivable from the others, and to show this, we will consider topological semantics for intuitionistic temporal logic.
As our axioms involve both $\diam$ and $\ubox$, we would like to be able to interpret both tenses.
Kremer observed that his semantics for $\ubox$ do not satisfy some key validities of $\bf LTL$, namely $\ubox p\to\tnext \ubox p$, $\ubox\circ p\to \tnext \ubox p$, and $\ubox p\to \ubox\ubox p$. This makes a proof-theoretic treatment of Kremer's logic difficult, as $\ubox \varphi \to \tnext \ubox \varphi$ is one of the defining properties of $\ubox$ and it is hard to tell what weaker principle could replace it.

To avoid this issue, we propose an alternative interpretation for $\ubox$. Our approach is natural from an algebraic perspective, as we define the interpretation of $\ubox \varphi$ via a greatest fixed point in the Heyting algebra of open sets. On the other hand, this fixed point is not definable in the classical language and hence we no longer obtain a sub-logic of $\bf DTL$. We will show that dynamic topological systems provide semantics for the logics without the constant domain axiom, from which we conclude the independence of the latter. Moreover, we show that the Fischer Servi axioms are valid for the class of {\em open} dynamical topological systems.

The constant domain axiom shows that the $\{\diam,\ubox\}$-logic of dynamic posets is different from that of dynamic topological systems. We show via an alternative axiom that the $\{\tnext,\ubox\}$-logics are also different. On the other hand, our main technical contribution is a proof that the $\{\tnext,\diam\}$-logics coincide, for which we use quasimodels, introduced in the context of intuitionistic temporal logics by Fern\'andez-Duque \cite{FernandezITLc}. This suggests that a completeness proof as in \cite{dtlaxiom} could be adapted to give a complete deductive calculus for the $\{\tnext,\diam\}$-logic over both the class of dynamic posets and the class of dynamic topological systems.

\paragraph{Layout.} Section \ref{SecBasic} introduces the syntax and the four axiomatic systems we propose for intuitionistic temporal logic. Section \ref{SecTopre} reviews dynamic topological systems, which are used in Section \ref{SecSemantics} to provide semantics for our formal language. Section \ref{SecSound} shows that each of the four logics is sound for a class of dynamical systems. These soundness results are used in Section \ref{SecInd} to show that the four logics are pairwise distinct. Section \ref{SecNDQ} reviews non-deterministic quasimodels, which are used in Section \ref{SecCons} to show that the topological and the Kripke $\{\tnext,\diam\}$-logics coincide. Finally, Section \ref{SecConc} lists some open questions.

\section{Syntax and axiomatics}\label{SecBasic}

In this section we will introduce four natural intuitionistic temporal logics. All of the axioms have appeared either in the intuitionistic logic, the temporal logic, or the intuitionistic modal logic literature. They will be based on the language of linear temporal logic, as defined next.

Fix a countably infinite set $\mathbb P$ of `propositional variables'. The language $\mathcal L$ of intuitionistic (linear) temporal logic $\bf ITL$ is given by the grammar
\[ \bot \  | \   p  \ |  \ \varphi\wedge\psi \  |  \ \varphi\vee\psi  \ |  \ \varphi\to\psi  \ |  \ \tnext\varphi \  | \  \diam\varphi \  |  \ \ubox\varphi, \]
where $p\in \mathbb P$. As usual, we use $\neg\varphi$ as a shorthand for $\varphi\to \bot$ and $\varphi \leftrightarrow \psi$ as a shorthand for $(\varphi \to \psi) \wedge (\psi \to \varphi)$. We read $\tnext$ as `next', $\diam$ as `eventually', and $\ubox$ as `henceforth'.
Given any formula $\varphi$,
we denote the set of subformulas of $\varphi$ by ${\mathrm{sub}}(\varphi)$ and its length by $|\varphi|$.
The language $\mathcal L_{\diam}$ is defined as the sublanguage of $\mathcal L$ without the modality $\ubox$.
Similarly, $\mathcal L_{\ubox}$ is the language without $\diam$.


We begin by establishing our basic axiomatization. 
It is obtained by adapting the standard axioms and inference rules of $\bf LTL$ \cite{temporal}, as well as their dual versions,
to propositional intuitionistic logic \citep{MintsInt}.
The logic $\logbasic$ is the least set of $\mathcal L$-formulas closed under the following rules and axioms.
\begin{enumerate}[itemsep=0pt,label=(\roman*)]
\item\label{ax01Taut} All intuitionistic tautologies.
\item\label{ax02Bot} $\neg \tnext \bot$
\item\label{ax03NexWedge} $\tnext \left( \varphi \wedge \psi \right) \leftrightarrow \left(\tnext \varphi \wedge\tnext \psi\right)$;
\item\label{ax04NexVee} $\tnext \left( \varphi \vee \psi \right) \leftrightarrow \left(\tnext \varphi \vee\tnext \psi\right)$;
\item\label{ax05KNext} $\tnext\left( \varphi \rightarrow \psi \right) \rightarrow \left(\tnext\varphi \rightarrow \tnext\psi\right)$;
\item\label{ax06KBox} $\ubox \left( \varphi \rightarrow \psi \right) \rightarrow \left(\ubox \varphi \rightarrow \ubox \psi\right)$;
\item\label{ax07:K:Dual} $\ubox \left( \varphi \rightarrow \psi \right) \rightarrow \left(\diam \varphi \rightarrow \diam \psi\right)$;
\item\label{ax08DiamVee} $\diam \left( \varphi \vee \psi \right) \rightarrow \left(\diam \varphi \vee\diam \psi\right)$;
\item\label{ax09BoxFix} $\ubox \varphi \to \varphi \wedge \tnext \ubox \varphi$;
\item\label{ax10DiamFix} $\varphi \vee \tnext \diam \varphi \to \diam \varphi$;
\item\label{ax11:ind:1} from ${ \varphi \rightarrow \tnext \varphi }$ infer ${ \varphi \rightarrow \ubox \varphi }$;
\item\label{ax12:ind:2} from ${ \tnext \varphi \to \varphi}$ infer ${ \diam \varphi \rightarrow \varphi } $;
\item\label{ax13MP} from $\varphi$ and $\varphi\to \psi$ infer $\psi$;
\item\label{ax14NecCirc} from $\varphi$ infer ${\tnext\varphi}$.
\end{enumerate}

However, modal intuitionistic logics typically involve additional axioms, due to Fischer Servi \cite{FS84}, in order to strengthen the ties with first-order intuitionistic logic. Thus we may also consider logics with the latter; for ${\rm FS}_\tnext$, recall that $\tnext$ is self-dual.
\medskip

\noindent $\begin{array}{ll}
({\rm FS}_\tnext (\varphi,\psi)) & \left(\tnext \varphi \rightarrow \tnext \psi \right) \to \tnext \left(\varphi \rightarrow \psi\right),\\	  
({\rm FS}_\diam (\varphi,\psi)) & \left( \diam \varphi \rightarrow \ubox \psi \right) \to \ubox \left(\varphi \rightarrow \psi\right).
\end{array}$
\medskip

Finally, we consider additional axioms reminiscent of constant domain axioms in first-order intuitionistic logic. As we will see, in the context of intuitionistic temporal logics, these axioms separate Kripke semantics from the more general topological semantics.
\medskip

\noindent $\begin{array}{ll}
(\AxCons \varphi\psi ) & \ubox( \varphi \vee \psi) \to \ubox \varphi \vee \diam \psi,\\	  
(\AxBInd \varphi\psi ) & \ubox ( \varphi \vee \psi) \wedge \ubox (\tnext \psi\rightarrow \psi) \rightarrow \ubox \varphi \vee \psi .
\end{array}$
\medskip

\noindent Here, $\rm CD$ stands for `constant domain' and $\rm BI$ for `backward induction'.
The axiom $\rm BI$ is meant to be a $\diam$-free approximation to $\rm CD$, as witnessed by the following.

\begin{proposition}\label{PropConstoBI} $\logbasic \vdash \AxCons pq \rightarrow \AxBInd pq.$
\end{proposition}

\proof
We reason within $\logbasic$.
Assume that
\begin{enumerate*}[label=(\arabic*)]
\item\label{ItDerivOne} $\AxCons pq$ holds, along with

\item\label{ItDerivTwo} $\ubox (\tnext q\rightarrow q)$ and
\item\label{ItDerivThree} $\ubox ( p\vee q)$.
\end{enumerate*}
From \ref{ItDerivOne} and \ref{ItDerivThree} we obtain $\ubox p \vee \diam q$, which together with \ref{ItDerivTwo} and axiom \ref{ax12:ind:2} gives us $\ubox p \vee q$, as needed.
\endproof

\noindent With this, we define the following logics:
\begin{align*}
  \loghomeo &\equiv \logbasic + {\rm FS}_\tnext + {\rm FS}_\diam, \\
  \logexp   &\equiv  \logbasic + {\rm CD},\\
  \logpers  &\equiv  \loghomeo +  \logexp.
\end{align*}
We are also interested in logics over sublanguages of $\mathcal L$.
For any logic $\Lambda $ defined above, let $\Lambda_\diam$ be the logic obtained by restricting all rules and axioms to $\mathcal L_\diam$, and let $\Lambda_\ubox$ be defined by restricting similarly to $\mathcal L_\ubox$, except that when $\rm CD$ is an axiom of $\Lambda$, we add the axiom $\rm BI$ to $\Lambda_\ubox$.

\section{Dynamic topological systems}\label{SecTopre}

The four logics defined above are pairwise distinct.
We will show this by introducing semantics for each of them.
They will be based on dynamic topological systems (or dynamical systems for short), which, as was observed in \cite{FernandezITLc}, generalize their Kripke semantics \cite{BoudouCSL}.

\subsection{Topological spaces and continuous functions}

Let us recall the definition of a {\em topological space}~\cite{munkres2000}:

\begin{definition}
A {\em topological space} is a pair $\< X ,\mathcal{T}\>,$ where $X$ is a set and $\mathcal T$ a family of subsets of $X$ satisfying
\begin{enumerate}[label=(\alph*)]
\item $\varnothing,X\in \mathcal T$;
\item if $U,V\in \mathcal T$ then $U\cap V\in \mathcal T$, and
\item if $\mathcal O\subseteq\mathcal T$ then $\bigcup\mathcal O\in\mathcal T$.
\end{enumerate}
The elements of $\mathcal T$ are called {\em open sets}.
\end{definition}

If $x \in X$, a {\em neighbourhood} of $x$ is an open set $U \subseteq X$ such that $x \in U$. Given a set $A\subseteq X$, its {\em interior}, denoted $A^\circ$, is the largest open set contained in $A$.
It is defined formally by
\begin{equation}\label{EqInterior}
A^\circ=\bigcup\cbra U\in\mathcal T:U\subseteq A\cket.
\end{equation}
Dually, we define the closure $\overline A$ as $X\setminus(X\setminus A)^\circ$; this is the smallest closed set containing $A$.

If $\<X,\mathcal T\>$ is a topological space, a function $S\colon X \to X$ is {\em continuous} if, whenever $U \subseteq X$ is open, it follows that $S^{-1}[U]$ is open.
The function $S$ is {\em open} if, whenever $V \subseteq X$ is open, then so is $S[V]$.
An open, continuous function is an {\em interior map}, and a bijective interior map is a {\em homeomorphism}.

A dynamical system is then a topological space equipped with a continuous function:

\begin{definition}
A {\em dynamical (topological) system} is a triple $\mathcal X = (X,\mathcal T,S)$ such that $(X,\mathcal T)$ is a topological space and $S\colon X \to X$ is continuous. We say that $\mathcal X$ is {\em invertible} if $S$ is a {\em homeomorphism,} i.e., $S^{-1}$ is also a continuous function, and {\em open} if $S$ is an interior map.
\end{definition}

\subsection{Up-set topologies}

Topological spaces generalize posets in the following way. Let $\mathcal F=\<W,{\acc}\>$ be a poset; that is, $W$ is any set and $\acc$ is a transitive, reflexive, antisymmetric relation on $W$.
To see $\mathcal F$ as a topological space, define $\mathord \uparrow w=\cbra v:w\peq v\cket.$
Then consider the topology $\mathcal T_\peq$ on $W$ given by setting $U\subseteq W$ to be open if and only if, whenever $w\in U$, we have $\mathord \uparrow w\subseteq U$. A topology of this form is a {\em up-set topology} \cite{alek}.
The interior operator on such a topological space can be computed by
\begin{equation}\label{EqIntPoset}
A^\circ = \{w \in W : {\uparrow} w\subseteq A\};
\end{equation}
i.e., $w$ lies on the interior of $A$ if whenever $v \seq w$, it follows that $v\in A$.

Throughout this text we will often identify partial orders with their corresponding topologies, and many times do so tacitly.
In particular, a dynamical system generated by a poset is called a \emph{dynamic poset}.
It will be useful to characterize the continuous and open functions on posets:
\begin{lemma}
  Consider a poset $\< W, \mathord\peq \>$ and a function $S\colon W \to W$. Then,
  \begin{enumerate}
  
\item    $S$ is continuous with respect to the up-set topology if and only if, whenever $w \peq w'$, it follows that $S(w) \peq S(w')$, and
  
\item    $S$ is open with respect to the up-set topology if whenever $S(w) \peq v$, there is $w' \in W$ such that $w \peq w'$ and $S(w') = v$.
    
  \end{enumerate}
\end{lemma}
These are confluence properties common in multi-modal logics; 
note that in \cite{BoudouCSL} we referred to maps satisfying the two conditions as {\em persistent maps.}

\begin{figure}

\begin{center}

\begin{tikzpicture}[scale=.7]

\def\y{1.3}

\draw (1.5,\y) node {Continuity};

\draw[thick] (0,0) circle (.35);

\draw (.04,0+.03) node {$w'$};

\draw[very thick,->] (.5,0) -- (2.5,0);

\draw (1.5,.5) node {$S$};

\draw[thick] (3,0) circle (.35);


\draw[thick] (0,-3) circle (.35);

\draw (0,-3) node {$w$};

\draw[very thick,-> ] (.5,-3) -- (2.5,-3);

\draw (1.5,-3.5) node {$S$};

\draw[very thick,->] (0,-2.5) -- (0,-.5);

\draw (-.5,-1.5) node {{\large$\rotatebox[origin=c]{90}{$\peq$}$}};

\draw[thick] (3,-3) circle (.35);


\draw[very thick,->,dashed] (3,-2.5) -- (3,-.5);

\draw (3.5,-1.5) node {{\large$\rotatebox[origin=c]{90}{$\peq$}$}};


\def\x{6}

\draw (1.5+\x,\y) node {Openness};

\draw[thick,dashed] (0+\x,0) circle (.35);

\draw (0+\x,0+.03) node {$w'$};

\draw[very thick,->,dashed] (.5+\x,0) -- (2.5+\x,0);

\draw (1.5+\x,.5) node {$S$};

\draw[thick] (3+\x,0) circle (.35);

\draw (3+\x,0) node {$v$};

\draw[thick] (0+\x,-3) circle (.35);

\draw (0+\x,-3) node {$w$};

\draw[very thick,->] (.5+\x,-3) -- (2.5+\x,-3);

\draw (1.5+\x,-3.5) node {$S$};

\draw[very thick,->,dashed] (0+\x,-2.5) -- (0+\x,-.5);

\draw (-.5+\x,-1.5) node {{\large$\rotatebox[origin=c]{90}{$\peq$}$}};

\draw[thick] (3+\x,-3) circle (.35);


\draw[very thick,->] (3+\x,-2.5) -- (3+\x,-.5);

\draw (3.5+\x,-1.5) node {{\large$\rotatebox[origin=c]{90}{$\peq$}$}};

\end{tikzpicture}

\end{center}

\caption{On a dynamic poset the above diagrams can always be completed if $S$ is continuous or open, respectively. Open, continuous maps on a poset are {\em persistent.}}\label{FigCO}
\end{figure}
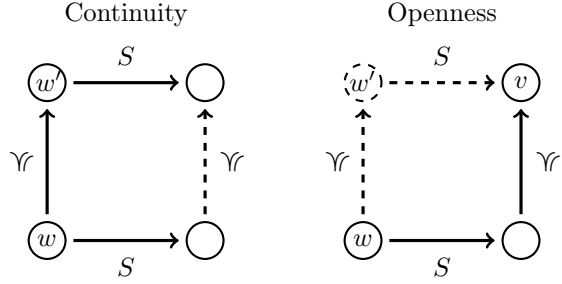


\section{Semantics}\label{SecSemantics}

In this section we will see how dynamical systems can be used to provide a natural intuitionistic semantics for the language of linear temporal logic.

\subsection{Basic definitions}

Formulas are interpreted as open subspaces of a dynamical system. Each propositional variable $p$ is assigned an open set $\val p$, and then $\val\cdot$ is defined recursively for more complex formulas according to the following:

\begin{definition}\label{DefSem}
Given a dynamical system $\mathcal X=(X,\mathcal T ,S)$, a {\em valuation on $\mathcal X$} is a function $\lb\cdot\rb\colon\mathcal L\to \mathcal T$ such that:
\begin{align*}
\lb\bot\rb &=\varnothing \\
\lb\varphi\wedge\psi\rb &=\lb\varphi\rb\cap \lb\psi\rb \\
\lb\varphi\vee\psi\rb &=\lb\varphi\rb\cup \lb\psi\rb\\
\lb\varphi\to\psi\rb &= \big ( (X\setminus\lb\varphi\rb)\cup \lb\psi\rb\big )^\circ\\
\val{\tnext\varphi}&=S^{-1} \val\varphi \\
\val{\diam\varphi}&=\textstyle\bigcup_{n\geq 0}S^{-n} \val\varphi \\
\val{\ubox\varphi} &= \bigcup \ \Big \{U\in \mathcal T : S[U] \subseteq U\subseteq \val\varphi \Big \}
\end{align*}
A tuple $\mathcal M = (X, \mathcal T, S, \val\cdot)$ consisting of a dynamical system with a valuation is a {\em dynamic topological model,} and if $\mathcal T$ is generated by a partial order, we will say that $\mathcal M$ is a {\em dynamic poset model.}
\end{definition}

All of the semantic clauses are standard from either intuitionistic or temporal logic, with the exception of that for $\ubox\varphi$, which we discuss in greater detail below. It is not hard to check by structural induction on $\varphi$ that $\val \varphi$ is uniquely defined given any assignment of the propositional variables to open sets, and that $\val \varphi$ is always open.

In practice, it is convenient to have a `pointwise' characterization of Definition \ref{DefSem}. For a model $\mathcal M = (X, \mathcal T, S, \val\cdot)$, $x \in X$ and $\varphi \in \mathcal L$, we write $\mathcal M,x \models \varphi$ if $x\in \val \varphi$, and $\mathcal M \models \varphi $ if $\val \varphi = X$.
Then, in view of \eqref{EqInterior}, given formulas $\varphi$ and $\psi$, $\mathcal M, x \models {\varphi \to \psi}$ if and only if there is a neighbourhood $U$ of $x$ such that for all $y \in U$, if $\mathcal M, y \models \varphi$ then $\mathcal M, y \ \models \psi$; note that this is a special case of {\em neighbourhood semantics} \cite{PacuitNeighborhood}.

Using \eqref{EqIntPoset}, this can be simplified somewhat in the case that $\mathcal T$ is generated by a partial order $\peq$:

\begin{proposition}
If $(X,{\peq},S,\val\cdot)$ is a dynamic poset model, $x\in X$, and $\varphi$, $\psi$ are formulas, then $\mathcal M, x \models {\varphi \to \psi}$ if and only if whenever $y \seq x$ and $\mathcal M, y \models \varphi$, it follows that $\mathcal M, y \models \psi$.
\end{proposition}

This is the standard relational interpretation of implication, and thus topological semantics are a generalization of the usual Kripke semantics.

\subsection{The topological `henceforth'}

Now let us discuss the topological interpretation of `henceforth', which is the main novelty in our semantics. In classical temporal logic, $\lb \ubox\varphi\rb$ is the largest set contained in $\lb\varphi\rb$ which is closed under $S$. In our semantics, $\lb\ubox\varphi\rb$ is the greatest {\em open} set which is closed under $S$. From this perspective, our interpretation is the natural intuitionistic variant of the classical one.
If $\mathcal M, x \models {\ubox \varphi}$, this fact is witnessed by an open, $S$\mbox{-}invariant neighbourhood of $x$, where $U\subseteq X$ is {\em $S$-invariant} if $S[U] \subseteq U$.

\begin{proposition}
If $(X,{\mathcal T},S,\val\cdot)$ is a dynamic topological model, $x\in X$, and $\varphi$ is any formula, then $\mathcal M, x \models {\ubox \varphi}$ if and only if there is an $S$-invariant neighbourhood $U$ of $x$ such that for all $y \in U$, $\mathcal M, y \models \varphi$.
\end{proposition}

In fact, the open, $S$-invariant sets form a topology; that is, the family of $S$-invariant open sets is closed under finite intersections and arbitrary unions. This topology is coarser than $\mathcal T$, in the sense that every $S$-invariant open set is (tautologically) open. Thus $\ubox$ can itself be seen as an interior operator based on a coarsening of $\mathcal T$, and $\val{\ubox\varphi}$ is always an $S$-invariant open set.

\begin{example}\label{ExBoxOnR}
As usual, the real number line is denoted by $\mathbb R$ and we assume that it is equipped with the standard topology, where $U \subseteq \mathbb R$ is open if and only if it is a union of intervals of the form $(a,b)$.
Consider a dynamical system based on $\mathbb R$ with $S \colon \mathbb R \to \mathbb R$ given by $S(x) = 2x$.
We claim that for any model $\mathcal M$ based on $(\mathbb R, S)$ and any formula $\varphi$, $\mathcal M, 0 \models {\ubox \varphi}$ if and only if $ \mathcal M \models \varphi$.

To see this, note that one implication is obvious since $\mathbb R$ is open and $S$-invariant, so if $\val\varphi = \mathbb R$ it follows that $ \mathcal M, 0 \models {\ubox\varphi}$.
For the other implication, assume that $\mathcal M, 0 \models {\ubox\varphi}$, so that there is an $S$-invariant, open $U\subseteq \val\varphi$ with $0 \in U$.
It follows from $U$ being open that for some $\varepsilon > 0$, $(-\varepsilon,\varepsilon) \subseteq U$.
Now, let $x \in \mathbb R$, and let $n$ be large enough so that $|2^{-n} x| < \varepsilon$.
Then, $2^{-n} x \in U$, and since $U$ is $S$-invariant, $x = S^n (2^{-n} x ) \in U$.
Since $x$ was arbitrary, $U = \mathbb R$, and it follows that $\mathcal M \models \varphi$.

On the other hand, suppose that $ 0 < a <x$ and $(a,\infty) \subseteq \val\varphi$.
Then, $(a,\infty)$ is open and $S$-invariant, so it follows that $x \in \val{\ubox\varphi}$.
Hence in this case we do not require that $\val\varphi = \mathbb R$. Similarly, if $x<a<0$ and $(-\infty,a) \subseteq \val \varphi$, we readily obtain $x \in \val{\ubox\varphi}$.
\end{example}

\subsection{The relational `henceforth'}

As was the case for implication, our interpretation for $\ubox$ becomes familiar when restricted to Kripke semantics.

\begin{lemma}\label{LemBoxKripke}
Let $\mathcal M = (W,{\peq},S,\val \cdot)$ be any dynamic poset model, $w\in W$ and $\varphi \in \mathcal L$. Then, the following are equivalent:
\begin{enumerate}[label=(\alph*)]

\item\label{ItOfficial} $\mathcal M, w \models \ubox \varphi$;

\item\label{ItKremer} $w \in \left ( \bigcap _{n<\omega} S^{-n} \val \varphi \right )^\circ$;

\item\label{ItKripke} for all $n<\omega$, $\mathcal M, S^n (w)  \models \varphi$.

\end{enumerate}
\end{lemma}

\proof
First we prove that \ref{ItOfficial} implies \ref{ItKremer}. Assume that $\mathcal M, w \models \ubox \varphi$, so that there is an $S$-invariant neighbourhood $U$ of $w$ with $U \subseteq \val \varphi$. To see that $w \in \left ( \bigcap _{n<\omega} S^{-n} \val \varphi \right )^\circ$, we must show that if $v \seq w$, then $v \in \bigcap _{n<\omega} S^{-n} \val \varphi$. So fix such a $v$ and $n<\omega$. Since $U$ is $S$-invariant, $S^n(w) \in U$, and since $S^n(v) \seq S^n(w)$ and $U$ is open, $S^n(v) \in U$, as needed. Thus $v \in \bigcap _{n<\omega} S^{-n} \val \varphi$, and since $v \seq w$ was arbitrary, \ref{ItKremer} holds.

That \ref{ItKremer} implies \ref{ItKripke} is immediate from
\[\left ( \bigcap _{n<\omega} S^{-n} \val \varphi \right )^\circ \subseteq \bigcap _{n<\omega} S^{-n} \val \varphi ,\]
so it remains to show that \ref{ItKripke} implies \ref{ItOfficial}.
Suppose that for all $n<\omega$, $\mathcal M, S^n (w)  \models \varphi$, and let $U = \bigcup_{n < \omega} {\uparrow} S^n(w)$. That the set $U$ is open follows from each ${\uparrow} S^n(w)$ being open and unions of opens being open.
If $v \in U$, then $v \seq S^n(w)$ for some $n<\omega$ and hence by upwards persistence, from $\mathcal M, S^n(w) \models \varphi$ we obtain $\mathcal M, v\models \varphi$; moreover, $S(v) \seq S^{n+1}(w)$ so $S(v) \in U$. Since $v \in U$ was arbitrary, we conclude that $U$ is $S$-invariant and $U\subseteq \val \varphi$.
Thus $U$ witnesses that $\mathcal M, w\models \ubox \varphi$.
\endproof

\begin{remark}
In fact, Kremer \cite{KremerIntuitionistic} uses \ref{ItKremer} as the definition of $\val {\ubox\varphi}$.
However, as we mentioned in the introduction, even our minimal axiomatic system $\logbasic$ is not sound for such an interpretation over arbitrary dynamical systems.
\end{remark}

\section{Soundness}\label{SecSound}

Recall that if $\mathcal M = (X,\mathcal T,S,\val\cdot)$ is any dynamic topological model and $\varphi \in \mathcal L$ is any formula, we write $\mathcal M \models \varphi$ if $\val\varphi =X$. Similarly, if $\mathcal X= (X,\mathcal T,S)$ is a dynamical system, we write $\mathcal X \models \varphi$ if for any valuation $\val \cdot$ on $\mathcal X$, we have that $(\mathcal X,\val\cdot) \models \varphi$. Finally, if $\Omega$ is a class of structures, we write $\Omega \models \varphi$ if for every $\mathcal A \in \Omega$, $\mathcal A \models \varphi$, in which case we say that $\varphi$ is {\em valid} on $\Omega$.

For purposes of this discussion, a {\em logic} may be any set $\Lambda \subseteq \mathcal L$, and we may write $\Lambda \vdash \varphi$ instead of $\varphi \in \Lambda$. Then, $\Lambda$ is {\em sound} for a class of structures $\Omega$ if, whenever $\Lambda \vdash \varphi$, it follows that $\Omega \models \varphi$.

In this section we will show that the four logics we have considered are sound for semantics based on different classes of dynamic topological systems (including dynamic preorders). 
For this, the following simple observation will be useful.

  \begin{lemma}\label{LemImpCrit} If $\mathcal M = (X, \mathcal T, S, \val\cdot)$ is any model and $\varphi,\psi \in \mathcal L$, then $\mathcal M \models \varphi \to \psi$ if and only if $\lb \varphi\rb \subseteq \lb \psi \rb$.   
  \end{lemma}
  
      \proof If $\lb \varphi\rb \subseteq \lb \psi \rb$ then $(X\setminus\lb\varphi\rb)\cup \lb\psi\rb =  X$, so
      $\lb\varphi\to\psi\rb =\big ( (X\setminus\lb\varphi\rb)\cup \lb\psi\rb\big )^\circ = X^\circ = X .$
    Otherwise, there is  $z \in \lb \varphi\rb $ such that $z \notin \lb \psi\rb$, so that $ z \notin \big ( (X\setminus\lb\varphi\rb)\cup \lb\psi\rb\big )^\circ$, i.e.~$z \notin  \lb\varphi \rightarrow \psi \rb $.
    \endproof

\subsection{Soundness of $\logbasic$}
    
    With this in mind, let us now show that our minimal logic is sound for the class of {\em all} dynamical systems.
   
\begin{theorem}\label{ThmSoundZero}
  $\logbasic$ is sound for the class of dynamical systems.
  \end{theorem}

  \proof
Let $\mathcal M = (X , \mathcal T, S, \val \cdot)$ be any dynamical topological model; we must check that all the axioms \ref{ax01Taut}-\ref{ax10DiamFix} are valid on $\mathcal M$ and all rules \ref{ax11:ind:1}-\ref{ax14NecCirc} preserve validity. Note that all intuitionistic tautologies are valid due to the soundness for topological semantics \cite{MintsInt}. Many of the other axioms can be checked routinely, so we focus only on those axioms involving the continuity of $S$ or the semantics for $\ubox$.
\medskip

\ignore{
\noindent \ref{ax02Bot}-\ref{ax04NexVee} These axioms follow from the fact that $S^{-1}[\cdot]$ is a Boolean homomorphism; for example, for \ref{ax03NexWedge},
\begin{align*}
&\lb \tnext \left( \varphi \wedge \psi \right) \rb = S^{-1}\val{\varphi \wedge \psi} = S^{-1} \big [\val{\varphi} \cap \val{\psi} \big ]\\
& = S^{-1} \val{\psi} \cap  S^{-1} \val{\psi} = \val { \tnext \varphi } \cap \val{ \tnext \psi } = \lb \tnext \varphi \wedge\tnext \psi \rb,
\end{align*}
so $\mathcal M \models \tnext (\varphi \wedge \psi) \leftrightarrow \tnext \varphi \wedge \tnext \psi$ in view of Lemma \ref{LemImpCrit}. Henceforth, we will use the latter lemma without mention.
\medskip
}

\noindent \ref{ax05KNext} Suppose that $x\in \val{\tnext(\varphi \to \psi)}$. Then, $S(x) \in \val {\varphi \to \psi}$. Since $S$ is continuous and $\val {\varphi \to \psi}$ is open, $U = S^{-1} \val{\varphi \to \psi}$ is a neighbourhood of $x$. Then, for $y \in U$, if $y \in \val{\tnext \varphi}$, it follows that $S(y) \in \val \varphi \cap \val {\varphi \to \psi}$, so that $S(y) \in \val \psi$ and $y \in \val {\tnext \psi}$. Since $y \in U$ was arbitrary, $x \in \val{\tnext\varphi \to \tnext \psi}$, thus $\val{\tnext(\varphi \to \psi)} \subseteq \val {\tnext \varphi \to \tnext \psi}$, and by Lemma \ref{LemImpCrit} (which from now on we will use without mention), \ref{ax05KNext} is valid on $\mathcal M$.
\medskip

\ignore{
\noindent \ref{ax06KBox} Suppose that $x\in \val{\ubox(\varphi \to \psi)}$. Then, there is an $S$-invariant neighbourhood $U$ of $x$ such that $U \subseteq \val{\varphi \to \psi}$. We claim that if $y \in U \cap \val{\ubox\varphi}$ it follows that $y \in \val{\ubox\psi}$, from which we obtain $x\in \val {\ubox\varphi \to \ubox \psi}$, as needed. If $y \in U \cap \val{\ubox\varphi}$, let $U'$ be an $S$-invariant neighbourhood of $y$ such that $U' \subseteq \val {\varphi}$, and define $V = U \cap U'$. Then, the set $V$ is an $S$-invariant neighbourhood of $y$. Moreover, if $z\in V$, then $z \in U \subseteq \val{\varphi \to \psi}$, while $z\in U' \subseteq \val \varphi$, hence $z\in \val \psi$. It follows that $V\subseteq \val \psi$, and thus $y \in \val {\ubox \psi}$, as desired.
\medskip
}

\noindent \ref{ax06KBox} Observe that $\val{\ubox(\varphi \to \psi)}$ is an $S$-invariant open subset of $\val{\varphi \to \psi}$.
Similarly, $\val{\ubox\varphi}$ is an $S$-invariant open subset of $\val{\varphi}$.
Let
\[U = \val{\ubox(\varphi \to \psi)} \cap \val{\ubox\varphi}.\]
Since $U$ is open, it suffices to prove that $U \subseteq \val{\ubox\psi}$.
Moreover, $U$ is $S$-invariant, therefore it suffices to prove that $U \subseteq \val{\psi}$,
which is direct because $U \subseteq \val{\varphi \to \psi} \cap \val{\varphi}$ and $\val{\varphi \to \psi} \subseteq (X \setminus \val{\varphi}) \cup \val{\psi}$.
\medskip

\ignore{
\noindent \ref{ax07:K:Dual} As before, suppose that $x\in \val{\ubox(\varphi \to \psi)}$, and let $U$ be an $S$-invariant neighbourhood of $x$ such that $U \subseteq \val{\varphi \to \psi}$. If $y \in U \cap \val{\diam \varphi}$, then $S^n(y)\in \val \varphi$ for some $n$; since $U$ is $S$-invariant, $S^n(y) \in U$, hence $S^n(y) \in \val {\psi}$ and $y \in \val{\diam \psi}$. We conclude that $x \in \val {\diam \varphi \to \diam \psi}$.
\medskip

\noindent \ref{ax08DiamVee} This axiom is standard from relational semantics. If $x \in \val {\diam (\varphi \vee \psi)}$, then $S^n(x) \in \val {\varphi \vee \psi}$ for some $n\geq 0$; it follows that either $S^n(x) \in \val \varphi$ or $S^n(x) \in \val \psi$, which in either case yields $x \in \val {\diam \varphi \vee \diam \psi}$.
\medskip
}

\noindent \ref{ax09BoxFix} Suppose that $x\in \val {\ubox\varphi}$, and let $U\subseteq \val \varphi$ be an $S$-invariant neighbourhood of $x$. Then, $x \in U$, so $x \in \val \varphi$. Moreover, $U$ is also an $S$-invariant neighbourhood of $S(x)$, so $S(x) \in \val {\ubox \varphi}$ and thus $x \in \val{\tnext \ubox \varphi}$. We conclude that $x \in \val{\varphi \wedge \tnext\ubox \varphi}$.
\medskip

\ignore{
\noindent \ref{ax10DiamFix} This axiom is also standard from relational semantics. If $x\in \val {\varphi \vee \tnext \diam \varphi}$, then either $x \in \val \varphi$ and hence $S^0(x) \in \val \varphi$, yielding $x\in \val {\diam \varphi}$, or $S(x) \in \val{\diam \varphi}$, yielding $S^{n+1}(x) = S^n(S(x)) \in \val \varphi$ for some $n$, hence $x\in \val {\diam \varphi}$.
\medskip
}

\noindent \ref{ax11:ind:1} If $\varphi \to \tnext \varphi$ is valid and $x\in \val \varphi$, then $\val \varphi$ is open (by the intuitionistic semantics) and $S$-invariant, since if $y \in \val \varphi$, from $y \in \val{\varphi \to \tnext\varphi}$ we obtain $S(y) \in \val \varphi$. It follows that $\val\varphi$ is an $S$-invariant neighbourhood of $x$, so $x\in \val {\ubox \varphi}$.
\ignore{
\medskip

\noindent \ref{ax12:ind:2} This rule is standard from relational semantics; suppose that $\tnext \varphi \to \varphi$ is valid, and let $x\in \val{ \diam \varphi}$. Let $n$ be least so that $S^n(x) \in \val \varphi$; if $n > 0$, from $S^{n-1}(x) \in \val {\tnext\varphi \to \varphi}$ we obtain $S^{n-1}(x) \in \val \varphi$, contradicting the minimality of $n$. Thus $n = 0$ and $x \in \val \varphi$.
\medskip
}
\ignore{
\noindent \ref{ax13MP}-\ref{ax14NecCirc} Follow from standard arguments.}
\endproof

\subsection{Soundness of stronger logics}

The additional axioms we have considered are valid over specific classes of dynamical systems. Specifically, the constant domain axiom is valid for the class of dynamic posets, while the Fischer Servi axioms are valid for the class of open systems. Let us begin by discussing the former in more detail.

\begin{theorem}\label{ThmSoundCD}
$\logexp$ and $\logexp_\ubox$ are sound for the class of dynamic posets.
\end{theorem}

\proof
Let $\mathcal M = (X,{\peq}, S, \val\cdot)$ be a dynamic poset model; in view of Theorem \ref{ThmSoundZero}, it only remains to check that $\rm CD$ and $\rm BI$ are valid on $\mathcal M$. However, by Proposition \ref{PropConstoBI}, $\rm BI$ is a consequence of $\rm CD$, so we only check the latter.
\medskip

\noindent (${\rm CD}(\varphi,\psi)$) \ Suppose that $x\in \val {\ubox (\varphi \vee \psi )}$, but $x \not \in \val {\ubox \varphi }$. Then, in view of Lemma \ref{LemBoxKripke}, for some $n \geq 0$, $S^n (x) \not \in \val{\varphi}$. It follows that $S^n (x) \in \val{\psi}$, so that $x \in \val {\diam \psi}$.
\endproof

Note that the relational semantics are used in an essential way, since Lemma \ref{LemBoxKripke} is not available in the topological setting. Now let's turn our attention to the Fischer Servi axioms.

\begin{theorem}\label{ThmSoundFS}
$\loghomeo$ is sound for the class of open dynamical systems.
\end{theorem}

\proof
Let $\mathcal M = (X,\mathcal T, S, \val \cdot)$ be a dynamical topological model where $S$ is an interior map. We check that axioms ${\rm FS}_\tnext$ and ${\rm FS}_\diam$ are valid on $\mathcal M$.
\medskip

\noindent (${\rm FS}_\tnext$) \ Suppose that $x \in \val {\tnext \varphi \to \tnext \psi}$, and let $U\subseteq \val {\tnext \varphi \to \tnext \psi}$ be a neighbourhood of $x$. Since $S$ is open, $V = S [U]$ is a neighbourhood of $S(x)$. Let $y \in V \cap \val \varphi$, and choose $z\in U$ so that $y = S(z)$. Then, $z\in U\cap \val {\tnext \varphi}$, so that $z\in \val {\tnext \psi}$, i.e.~$y\in \val \psi$. Since $y\in V$ was arbitrary, $S(x) \in \val{\varphi \to \psi}$, and $x\in \val{\tnext (\varphi \to \psi)}$.
\medskip

\noindent (${\rm FS}_\diam$) \ Suppose that $x \in \val {\diam \varphi \to \ubox \psi}$, and let $U\subseteq \val{\diam \varphi \to \ubox \psi}$ be a neighbourhood of $x$. Set $V = \bigcup _{n <\omega} S^n[U]$; since $S$ is open and unions of opens are open, $V$ is open as well. Moreover, $V$ is clearly $S$-invariant, as if $x \in V$, then $x \in S^n[U]$ for some $n\geq 0$, so that $S(x) \in S^{n+1}[U] \subseteq V$.

We claim that $V\subseteq \val{\varphi \to \psi}$, from which we obtain a witness that $\mathcal M, x \models {\ubox(\varphi \to \psi)}$. Suppose that $y \in V \cap \val \varphi$. By the definition of $V$, $y = S^n ( z )$ for some $n < \omega$ and some $z \in U$. Then, $z \in U \cap \val{\diam \varphi}$, so that $z\in \val {\ubox \psi}$. From this we may choose an $S$-invariant neighbourhood $Z \subseteq \val \psi$ of $z$. But $y = S^n(z) \in Z$ so that $ y \in \val \psi$, and since $y \in V$ was arbitrary we see that $V \subseteq \val{\varphi \to \psi}$, as needed.
\endproof

As an easy consequence, we mention the following combination of Theorems \ref{ThmSoundCD} and \ref{ThmSoundFS}. Recall that dynamic posets with an interior map are also called {\em persistent.}

\begin{corollary}\label{CorSoundOne}
$\logpers$ and $\logpers_\ubox$ are sound for the class of persistent dynamic posets.
\end{corollary}

\section{Independence}\label{SecInd}

In this section we will use our soundness results to show that the four logics we have considered are pairwise distinct.

\subsection{Independence of the constant domain axioms}

The formulas $\rm CD$ and $\rm BI$ separate Kripke semantics from the general topological semantics.

\begin{proposition}\label{PropInd}
The formulas $\AxCons pq$ and $\AxBInd pq$ are not valid over the class of invertible dynamical systems based on~$\mathbb R$.
\end{proposition}

\proof
Define a model $\mathcal M$ on $\mathbb R$, with $S(x) =  2x$, $\lb p \rb =  (-\infty,1)$ and $\lb q \rb = (0,\infty)$.
Clearly $\val{p\vee q} = \mathbb R$, so that $\val{\ubox(p \vee q)} = \mathbb R$ as well.

Let us see that $ \mathcal M, 0 \not \models \AxCons pq $.
Since $\mathcal M, 0 \models {\ubox(p \vee q)}$, it suffices to show that $\mathcal M, 0 \not \models {\ubox p \vee \diam q}$. It is clear that $\mathcal M, 0 \not \models {\diam q}$ simply because $S^n(0) = 0 \not \in \val q$ for all $n$. Meanwhile, by Example \ref{ExBoxOnR}, $\mathcal M, 0 \models {\ubox p}$ if and only if $\val p = \mathbb R$, which is not the case. We conclude that $
\mathcal M, 0 \not \models \AxCons pq.
$

To see that $
\mathcal M, 0 \not \models \AxBInd pq
$
we proceed similarly, where the only new ingredient is observing that $\mathcal M, 0  \models \ubox (\tnext q\rightarrow q)$. But this follows easily from the fact that if $\mathcal M,x \models \tnext q$, then $x > 0$ so that $\mathcal M,x \models q$, hence $\val{\tnext q \to q} = \mathbb R$.
\endproof

\begin{corollary}\label{CorCDIndep}
$\loghomeo \not \vdash \AxCons pq$ and $\loghomeo \not \vdash \AxBInd pq$.
\end{corollary}

\proof
By Theorem \ref{ThmSoundFS}, $\loghomeo$ is sound for the class of open dynamical systems, but by Proposition \ref{PropInd}, $\AxCons pq$ is not valid on this class, hence $\loghomeo_\ubox \not\vdash \AxCons pq$. That $\loghomeo \not \vdash \AxBInd pq$ is obtained by the same reasoning.
\endproof

\ignore{

\begin{proposition}\label{PropBItoCons}
The formula $\AxBInd pq \to \AxCons pq$ is not valid over the class of invertible dynamical systems based on $\mathbb R$.
\end{proposition}

\proof
Consider a model $\mathcal M$ similar to that used in the proof of Proposition \ref{PropInd}, except that $\lb q \rb = \mathbb R \setminus [-\nicefrac 12,\nicefrac 12].$ Then, $\ubox(p \vee q)  \rightarrow \ubox p \vee \diam q$ fails at $0$ (by essentially the same reasoning).
However, it could easily be checked that $\val{\ubox(\tnext q \rightarrow q)} = \val q$.
Hence $0 \in \lb \neg \ubox (\tnext q\rightarrow q) \rb$, from which it readily follows that $0$ satisfies \[\ubox (\tnext q\rightarrow q) \rightarrow \big (\ubox ( p\vee q) \rightarrow  \ubox p \vee q \big).\]
Therefore $\AxBInd pq$ does not imply $\AxCons pq$ over the class of invertible dynamical systems.
\endproof

Note, however, that Proposition \ref{PropBItoCons} does not necessarily imply that there are no formulas $\varphi$, $\psi$ such that $\AxBInd \varphi\psi \to \AxCons pq$ is derivable, and hence it is reasonable to use $\rm BI$ in place of $\rm CD$ to axiomatize $\diam$-free logics.
}

\subsection{Independence of the Fischer Servi axioms}

The Fischer Servi axioms are also not valid in general, as shown in Boudou et al.~\cite{IMLA} (see Figure \ref{FigIMLA}).

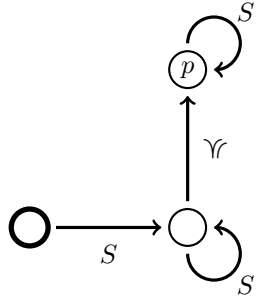
\begin{figure}

\begin{center}

\begin{tikzpicture}[scale=.7]

\def\y{1.3}





\draw[thick] (3,0) circle (.35);

\draw (3 , 0 ) node {$p$};

\draw[line width = 2] (0,-3) circle (.35);


\draw[very thick,-> ] (.5,-3) -- (2.5,-3);

\draw (1.5,-3.5) node {$S$};



\draw[thick] (3,-3) circle (.35);


\draw[very thick,->] (3,-2.5) -- (3,-.5);

\draw (3.5,-1.5) node {{\large$\rotatebox[origin=c]{90}{$\peq$}$}};

\draw[very thick,->] (3,-3.5) arc (-180:+90:.5) ;
 
\draw (4.1,-4.1) node {$S$};

\draw[very thick,->] (3,.5) arc (180:-90:.5) ;
 
\draw (4.1,1.1) node {$S$};

\end{tikzpicture}

\end{center}

\caption{A dynamic poset model falsifying both Fischer Servi axioms. Propositional variables that are true on a point are displayed; only one point satisfies $p$ and no point satisfies $q$. It can readily be checked that ${\rm FS}_\tnext(p,q)$ and ${\rm FS}_\diam(p,q)$ fail on the highlighted point on the left. Note that $S$ is continuous but not open, as can easily be seen by comparing to Figure \ref{FigCO}.}\label{FigIMLA}
\end{figure}

\begin{proposition}\label{prop:nvalid:iltl}
${\rm FS}_\tnext(p,q)$ and ${\rm FS}_\diam (p,q)$ are not valid over the class of dynamic posets.
\end{proposition}

From this and the soundness of $\loghomeo$ (Theorem \ref{ThmSoundFS}), we immediately obtain that they are not derivable in $\logbasic$.

\begin{corollary}\label{CorFSIndep}
$\logexp \not \vdash {\rm FS}_\tnext(p,q)$ and $\logexp \not \vdash {\rm FS}_\diam (p,q)$.
\end{corollary}

\ignore{
\proof
By Theorem \ref{ThmSoundFS}, $\loghomeo$ is sound for the class of open dynamical systems, but by Proposition \ref{PropInd}, $\AxCons pq$ is not valid on this class, hence $\loghomeo_\ubox \not\vdash \AxCons pq$. That $\loghomeo \not \vdash \AxBInd pq$ is obtained by the same reasoning.
\endproof
}

The above independence results are sufficient to see that our four logics are distinct.

\begin{theorem}\label{TheoDistinct}
The logics $\logbasic$, $\loghomeo$, $\logexp$ and $\logpers$ are pairwise distinct, as are $\logbasic_\ubox$, $\loghomeo_\ubox$, $\logexp_\ubox$ and $\logpers_\ubox$.
\end{theorem}

\proof
By Corollary \ref{CorCDIndep} and the definition of $\logexp$, $\AxCons pq \in \logexp \setminus \loghomeo$; similarly, by Corollary \ref{CorFSIndep}, ${\rm FS}_\tnext (p,q) \in \loghomeo \setminus \logexp$. Thus $\loghomeo$ and $\logexp$ are incomparable, from which we conclude that $\logbasic$, which is contained in their intersection, is strictly smaller than either of them, while $\logpers$, which contains their union, is strictly larger. The argument for the logics over $\mathcal L_\ubox$ are analogous, except that $\rm CD$ is replaced with $\rm BI$.
\endproof

\section{Types and quasimodels}\label{SecNDQ}

In this section we review non-deterministic quasimodels \cite{FernandezITLc}.
Quasimodels will be our fundamental tool for passing from topological to Kripke semantics.

\subsection{Two-sided types}

Our presentation will differ slightly from that of \cite{FernandezITLc}, since it will be convenient for us to use two-sided types, defined as follows.

\begin{definition}\label{def:type}
  Let $\Phi^-,\Phi ^ + \subseteq \mathcal L_\diam$ be finite sets of formulas. We say that the pair $\Phi = (\Phi ^- ; \Phi ^+) $ is a {\em two-sided type} if:
  \begin{enumerate}

  \item $\Phi^- \cap \Phi ^ +  = \varnothing$,
        \label{cond:type:intersection}

  \item $\bot\not\in \Phi ^ +$,
        \label{cond:type:bot}

  \item if $\varphi\wedge\psi\in \Phi ^ +$, then $\varphi,\psi\in \Phi^+$,
        \label{cond:type:posconj}

  \item if $\varphi\wedge\psi\in \Phi ^ -$, then  $\varphi \in \Phi ^-$ or $\psi\in \Phi^-$,
        \label{cond:type:negconj}

  \item if $\varphi\vee\psi\in \Phi ^ +$, then $\varphi \in \Phi^+$ or $\psi\in \Phi^+$,
        \label{cond:type:posdisj}

  \item if $\varphi\vee\psi\in \Phi ^ -$, then  $\varphi , \psi\in \Phi^-$,
        \label{cond:type:negdisj}

  \item if $\varphi\to\psi\in \Phi^+$, then either $\varphi \in \Phi^-$ or $\psi \in\Phi^+$, and
        \label{cond:type:implication}

\item\label{cond:type:diam} if $\diam \varphi \in \Phi^-$ then $\varphi \in \Phi^-$.

  \end{enumerate}
  The set of finite two-sided types will be denoted $\type{}$.
  Whenever $\Xi$ is an expression denoting a two-sided type,
  we write $\Xi^-$ and $\Xi^+$ to denote its components.
\end{definition}

We will consider two partial orders on $\type{}$. We will write
\begin{enumerate}[label=(\alph*)]

\item $\Phi \peqT \Psi$ if $\Psi^-\subseteq \Phi^-$ and $\Phi^+ \subseteq \Psi^+$, and

\item $\Phi \sqsubT \Psi$ if $\Phi^- = \Psi^-$ and $\Phi^+\subseteq \Psi^+$.

\end{enumerate}

If $\Phi$ is a two sided-type and $\Sigma = \Phi^- \cup \Phi^+$, we may say that $\Phi$ is a {\em two-sided $\Sigma$-type.}
The set of two-sided $\Sigma$-types will be denoted by $\type\Sigma$.

\begin{remark}\label{RemarkTypes}
Fern\'andez-Duque \cite{FernandezITLc} uses one-sided $\Sigma$-types, but it is readily checked that a one-sided type $\Phi$ as defined there can be regarded as a two-sided type $\Psi$ by setting $\Psi^+=\Phi$ and $\Psi^- = \Sigma \setminus \Phi$.
Henceforth we will write {\em type} instead of {\em two-sided type} and explicitly write {\em one-sided type} when discussing \cite{FernandezITLc}.
\end{remark}

\subsection{Quasimodels}

Quasimodels are similar to models, except that valuations are replaced with a labelling function $\ell$. We first define the more basic notion of {\em labelled frame.}

\begin{definition}\label{frame}
  A {\em labelled frame} is a triple $\mathcal F= ( W,{\peq},\ell )$,
  where $\peq$ is a partial order on $W$
  and $\ell\colon W\to \type{}$ is such that 
  \begin{enumerate}[label=(\alph*)]
    \item whenever $w \peq v$ it follows that $\ell(w) \peqT \ell(v)$, and
          \label{cond:frame:monotony}
    \item whenever $\varphi\to\psi \in \ell^-(w)$, there is $v \seq w$ such that $\varphi\in \ell^+(v)$ and $\psi \in \ell^-(v)$,
  \end{enumerate}
  where $(\ell^-(v), \ell^+(v)) \eqdef \ell(v)$.
 
 We say that $\mathcal F$ {\em satisfies} $\varphi\in\mathcal L$ if $\varphi\in \ell^+(w)$ for some $w\in W$,
  and that it {\em falsifies} $\varphi$ if $\varphi \in \ell^-(w)$ for some $w \in W$.
  If $\ell (w) \in \type \Sigma$ for all $w\in W$, we say that $\mathcal F$ is a {\em $\Sigma$-labelled frame.}
\end{definition}

Labelled frames model only the intuitionistic aspect of the logic.
For the temporal dimension, let us define a new relation over types.

\begin{definition}\label{compatible}
  We define a relation $\ST \subseteq \type{} \times \type{}$ by $\Phi \ST \Psi$ iff for all $\varphi \in \mathcal L$:
  \begin{enumerate}[label=(\alph*)]
  \item\label{ItCompOne} if $\tnext\varphi\in \Phi^+$ then $ \varphi\in \Psi^+$,
  \item\label{ItCompTwo} if $\tnext\varphi\in \Phi^-$ then $ \varphi\in \Psi^-$,
  \item\label{ItCompThree} if $\diam\varphi\in \Phi^+$, then $\varphi\in\Phi^+$ or $\diam\varphi\in \Psi^+,$ and
  \item\label{ItCompFour} if $\diam\varphi\in \Phi^-$, then $\diam \varphi \in \Psi^-$.
  \end{enumerate}
\end{definition}

Quasimodels are then defined as labelled frames with a suitable binary relation.

\begin{definition} \label{def:quasimodel}
  A \emph{quasimodel} is a tuple $\mathcal Q = (W, \mathord{\peq}, S, \ell)$
  where $(W, \mathord{\peq}, \ell)$ is a labelled frame
  and $S$ is a binary relation over $W$ that is
  \begin{description}
  
    \item[{\em serial:}] for all $w$ there is $v$ such that $w \mathrel S v$;
    \item[{\em forward-confluent:}] if $w \peq w'$ and $w \mathrel S v$, there is $v'$ such that $v\peq v'$ and $w'\mathrel S v'$;
    \item[{\em sensible:}] $w \mathrel S x$ implies $\ell(w) \ST \ell(x)$, and
    \item[{\em $\omega$-sensible:}] whenever $\diam\varphi\in \ell^+(w)$, there are $n\geq 0$ and $v$ such that $w \mathrel S^n v$ and $\varphi\in \ell^+(v)$.

  \end{description}
  If $(W, \mathord{\peq}, \ell)$ is a $\Sigma$-labeled frame then $\mathcal Q$ is a {\em $\Sigma$-quasimodel.}
  If $S$ is a function then $\mathcal Q$ is {\em deterministic.}
\end{definition}

The forward confluence condition plays the role of continuity in the non-deterministic setting; indeed, if $S$ is deterministic, then it is easy to see that $S$ is forward-confluent if and only if it is monotone, which as we have discussed, is equivalent to continuity with respect to the up-set topology.
In fact, deterministic quasimodels are essentially dynamic posets with a particular valuation, as witnessed by the following version of the `truth lemma':

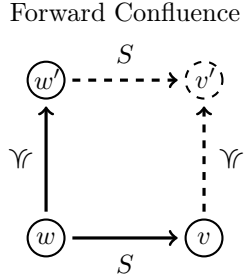
\begin{figure}

\begin{center}

\begin{tikzpicture}[scale=.7]

\def\y{1.3}

\draw (1.5,\y) node {Forward Confluence};

\draw[thick] (0,0) circle (.35);

\draw (.03,0+.02) node {$w'$};

\draw[very thick,->,dashed] (.5,0) -- (2.5,0);

\draw (1.5,.5) node {$S$};

\draw[thick,dashed] (3,0) circle (.35);

\draw (3+.03,.02+.02) node {$v'$};

\draw[thick] (0,-3) circle (.35);

\draw (0,-3) node {$w$};

\draw[very thick,-> ] (.5,-3) -- (2.5,-3);

\draw (1.5,-3.5) node {$S$};

\draw[very thick,->] (0,-2.5) -- (0,-.5);

\draw (-.5,-1.5) node {{\large$\rotatebox[origin=c]{90}{$\peq$}$}};

\draw[thick] (3,-3) circle (.35);

\draw (3,-3) node {$v$};

\draw[very thick,->,dashed] (3,-2.5) -- (3,-.5);

\draw (3.5,-1.5) node {{\large$\rotatebox[origin=c]{90}{$\peq$}$}};

\end{tikzpicture}

\end{center}

\caption{If $S$ is forward-confluent, then the above diagram can always be completed.}
\end{figure}

\begin{lemma}\label{LemTruth}
  Let ${\mathcal Q}=(W,{\peq},S,\ell)$ be a deterministic quasimodel, and define a valuation $\val\cdot^{\ell}$ on $ {\mathcal Q}$ by setting
$\val p^\ell =\{w \in W : p \in \ell^+(w)\}$
and extending to all of $\mathcal L$ recursively.
Then, for all formulas $\varphi \in \mathcal L_\diam$ and for all $w \in { W}$,
	\begin{enumerate}[label=(\arabic*)]
		\item if $\varphi \in { \ell}^+(w)$ then  $w \in \val{\varphi}^{\ell}$, and
		\item if $\varphi \in { \ell}^-(w)$ then  $w \not \in \val{\varphi}^{\ell}$.
	\end{enumerate}
\end{lemma}

\begin{proof}
We proceed by structural induction on $\varphi$. We must consider the following cases.
\medskip

\noindent ($\varphi = p$ is an atom) \ Note that by definition of $\val{p}^{\ell}$, if $p \in {\ell}^+(w)$ then $w \in \val{p}^{\ell}$ and if $p \in {\ell}^-(w)$ then $p \not \in {\ell}^+(w)$ so $w \not \in \val{p}^{\ell}$.
\medskip

\noindent ($\varphi = \psi \wedge \theta$) \ Assume that $\psi \wedge \theta \in {\ell}^+(w)$. By Definition~\ref{def:type} it follows that $\psi \in {\ell}^+(w)$ and $\theta \in {\ell}^+(w)$. By induction hypothesis,  $w \in \val{\psi}^{\ell}$ and $w \in \val{\theta}^{\ell}$. Therefore $w \in \val{\psi \wedge \theta}^{\ell}$.

If $\psi \wedge \theta \in {\ell}^-(w)$, by definition~\ref{def:type} it follows that either $\psi \in {\ell}^-(w)$ or $\theta \in {\ell}^-(w)$. By induction hypotheses we conclude that $w \not \in \val{\psi}^{\ell}$ or $w \not \in \val{\theta}^{\ell}$. Therefore $w \not \in \val{\psi \wedge \theta}^{\ell}$.
\medskip

\noindent ($\varphi = \psi \vee \theta$) This case is symmetric, but using the conditions for $\vee$.
\medskip

\noindent ($\varphi = \psi \to \theta$) \
Assume first that $\psi \to \theta \in {\ell}^+(w)$.
Then for all $y$ such that $w \peq y$, by condition~\ref{cond:frame:monotony} of Definition~\ref{frame}, $\psi \to \theta \in \ell^+(y)$.
By condition~\ref{cond:type:implication} of Definition~\ref{def:type} and by induction hypothesis, $y \notin {\val \psi}^\ell$ or $y \in {\val \theta}^\ell$.
Therefore, $w \in {\val{\psi \to \theta}}^\ell$.

Now let us assume that $\psi \to \theta \in {\ell}^-(w)$. By Definition~\ref{frame} it follows that there exists $v \in W$ such that $w\peq v$ and $\psi \in {\ell}^+(v)$ and $\theta \in {\ell}^-(v)$. By induction hypothesis it follows that $v  \in \val{\psi}^{\ell} \setminus \val{\theta}^{\ell}$, which means that $w \not \in \val{\psi\to \theta}^{\ell}$.
\medskip

\noindent ($\varphi = \tnext \psi$) \ Assume that $\circ \psi \in {\ell}^+(w)$. Since $S$ is sensible, $\psi \in {\ell}^+(S(w))$. By induction hypothesis $S(w) \in \val{\psi}^{\ell}$. Therefore $w \in \val{\circ \psi}^{\ell}$. The case where $\circ \psi \in {\ell}^-(w)$ is analogous.
\medskip

\noindent ($\varphi = \diam \psi$) If $\diam \psi \in {\ell}^+(w)$, by the fact that $S$ is $\omega$-sensible there exists $v\in W$ such that $w \mathrel S^n v$ and $\psi \in {\ell}^+(v)$; since $S$ is deterministic, we must forcibly have $v=S^n(w)$. By induction hypothesis we conclude that $v \in \val{\psi}^{\ell}$ and by the satisfaction relation it follows that $w \in \val{\diam \psi}^{\ell}$.

In case that $\diam \psi \in {\ell}^-(w)$, observe that for all $n$, if $\diam \psi \in {\ell}^-(S^n(w))$ then $\diam \psi \in {\ell}^-(S^{n+1}(w))$; thus by induction, $\diam \psi \in {\ell}^-(S^n(w))$ for all $n<\omega$. In virtue of Definition \ref{def:type}.\ref{cond:type:diam}, $\psi \in {\ell}^-(S^n(w))$ for all $n<\omega$, hence by the induction hypothesis $S^n(w) \not \in \val \psi$, from which it follows that $w \not \in \val{\diam \psi}$.
\end{proof}

In the non-deterministic case quasimodels are not models as they stand, but in \cite{FernandezITLc}, it is shown that dynamical systems can be extracted from them.

\begin{theorem}\label{TheoITLc}
A $\ubox$-free formula $\varphi$ is satisfiable (falsifiable) over the class of dynamic topological systems if and only if it is satisfiable (falsifiable) over the class of finite, ${\rm sub}(\varphi)$-quasimodels.
\end{theorem}

Theorem \ref{TheoITLc} applies even to formulas in an extended language $\mathcal L_{\diam\forall}$ with the universal modality, but it is shown in \cite{FernandezITLc} that there are topologically falsifiable $\mathcal L_{\diam\forall}$-formulae that are not Kripke falsifiable.
As we will see, this is not the case over $\mathcal L_{\diam}$.
Note that \cite{FernandezITLc} uses quasimodels with one-sided types, but in view of Remark \ref{RemarkTypes}, the theorem can easily be modified to obtain quasimodels with two-sided types.

\subsection{Restrictions on types}

The reason that two-sided types are convenient is that they can easily be restricted to smaller sets of formulas while maintaining the relations between them. To make this precise, if $\Sigma$ is a set of formulas, first define $\Psi\upharpoonright \Sigma = (\Psi^- , \Psi^+\cap \Sigma)$, and ${\rm sub} (\Sigma) = \bigcup _{\varphi \in \Sigma}{\rm sub}(\varphi)$. With this, we have the following:

\begin{lemma}\label{LemmRestrict}
Let $\Phi,\Psi,\Gamma,\Theta$ be types and $\Sigma$ a set of formulas closed under subformulas. Then,
\begin{enumerate}

\item\label{cond:alsotype} $\Phi\upharpoonright \Sigma$ is also a type;

\item\label{ItCrossTrans} if $\Gamma \sqsubT \Phi \peqT \Psi$ or $\Gamma \peqT \Phi \sqsubT \Psi$ then $\Gamma \peqT \Psi$, and


\item\label{ItRestrictSqsub} if $\Gamma \sqsubT \Phi \ST \Psi$ and ${\rm sub} ( \Gamma^+ ) \subseteq \Sigma$, then $\Gamma \ST \Psi \upharpoonright \Sigma$.


\end{enumerate}
\end{lemma}

\proof
	
To prove item \ref{cond:alsotype} it is sufficient to check that the conditions of Definition~\ref{def:type} hold. 
Conditions \ref{cond:type:intersection} and~\ref{cond:type:bot} of Definition~\ref{def:type} are straightforward.
Since $\Phi^- = (\Phi\upharpoonright\Sigma)^-$, conditions \ref{cond:type:negconj} and~\ref{cond:type:negdisj} clearly hold.
For condition~\ref{cond:type:implication}, suppose that $\varphi\to\psi\in  (\Psi\upharpoonright \Sigma)^+ $.
Since $\Sigma$ is closed under subformulas, $\varphi, \psi \in \Sigma$ and, since $\Psi$ is a type it follows that either $\varphi \in \Psi^-$ or $\psi \in \Psi^+$.
By definition either $\varphi \in \Psi^-$ or $\psi \in \Psi^+ \cap \Sigma$.
The proofs for conditions \ref{cond:type:posconj} and~\ref{cond:type:posdisj} of Definition~\ref{def:type} are similar and left to the reader. 

Regarding item~\ref{ItCrossTrans} of the lemma, on one side, $\Gamma \sqsubT \Phi \peqT \Psi$ means that $\Gamma^+ \subseteq \Phi^+ \subseteq \Psi^+$ and $\Gamma^- = \Phi^- \supseteq \Psi^-$. Therefore $\Gamma^+ \subseteq \Psi^+$ and $\Psi^- \subseteq \Gamma^-$ so $\Gamma \peqT \Psi$. On the other side  $\Gamma \peqT \Phi \sqsubT \Psi$ means by definition that $\Gamma^+ \subseteq \Psi^+ \subseteq \Psi^+$ and $\Gamma^- \supseteq \Phi^- = \Psi^-$. It follows that $\Gamma^+ \subseteq \Psi^+$ and $\Psi^+ \subseteq \Gamma^-$ so $\Gamma \peqT \Psi$.

For item~\ref{ItRestrictSqsub} we consider the conditions of Definition~\ref{compatible}:\medskip

\noindent \ref{ItCompOne} If $\tnext \psi \in \Gamma^+$, from $\Gamma \sqsubT  \Phi\ST\Psi$ we conclude that $\tnext \psi \in \Phi^+$ and $\psi \in \Psi^+$. Since ${\rm sub}(\Gamma^+) \subseteq \Sigma$ then $\psi \in \Sigma$. Therefore $\psi \in \Psi^+ \cap \Sigma$ so $\psi \in \left(\Psi \upharpoonright \Sigma\right)^+$.\medskip

\noindent \ref{ItCompTwo} If $\tnext \psi \in \Gamma^-$, from $\Gamma \sqsubT  \Phi\ST\Psi$ we conclude that $\tnext \psi \in \Phi^-$ and $\psi \in \Psi^-$, which by definition means that $\psi \in \left(\Psi \upharpoonright \Sigma  \right)^-$.
\medskip

\noindent \ref{ItCompThree} If $\diam \psi \in \Gamma^+$, since ${\rm sub}(\Gamma^+) \subseteq \Sigma$ then $\diam \psi, \psi \in \Sigma$. From $\Gamma \sqsubT  \Phi\ST\Psi$ we conclude that $\diam \psi \in \Phi^+$ and either $\psi \in \Gamma^+$ or $\diam \psi \in \Psi^+$. From this it follows that either $\psi \in \Gamma^+$ or  $\diam \psi \in \Psi^+ \cap \Sigma$ (which means that $\diam \psi \in \left(\Psi \upharpoonright \Sigma\right)^+$).
\medskip

\noindent \ref{ItCompFour} If $\diam \psi \in \Gamma^-$, from $\Gamma \sqsubT  \Phi\ST\Psi$ we conclude that $\diam \psi \in \Phi^-$, $\psi \in \Phi^-$ (thus $\psi \in \Gamma^-$) and $\diam \psi \in \Psi^-$. As a consequence it follows that $\psi \in \Gamma^-$ and $\diam \psi \in \left(\Psi \upharpoonright \Sigma\right)^-$.
\endproof

We may also wish to `forget' temporal formulas that have been realized. To make this precise, let ${\rm sup}(\varphi)$ denote the set of {\em super}-formulas of $\varphi$. Say that a formula $\varphi$ is a {\em temporal formula} if it is of the forms $\tnext \psi$ or $\diam \psi$, and if $\Phi$ is a set of formulas, say that $\varphi \in \Phi$ is {\em maximal in $\Phi$} if it does not have any temporal superformulas in $\Phi$. Then, define $\remove\Phi\varphi =(\Phi^-,\Phi^+\setminus {\rm sup}(\varphi))$.

\begin{lemma}\label{LemDelete}
	Suppose that $\Phi \ST \Psi$.
	\begin{enumerate}
		
		\item\label{ItDeleteCirc}  If $\tnext\varphi$ is maximal in $\Phi^+$, then $\Phi \ST \remove \Psi {\tnext\varphi}$ .
		
		\item\label{ItDeleteDiam}  If $\diam \varphi $ is maximal in $\Phi^+$ and $\varphi \in \Phi^+$, then $\Phi \ST \remove \Psi {\diam \varphi}$.
		
	\end{enumerate}
\end{lemma}

\proof
We prove the first item; the second is analogous. Assuming that $\tnext\varphi$ is maximal in $\Phi^+$, let us check that the four conditions of Definition~\ref{compatible} hold.
\medskip

\noindent\ref{ItCompOne} If $\tnext \theta \in \Phi^+$, since $\Phi \ST \Psi$ then $\theta \in \Psi^+$. Moreover, since $\tnext \theta \in \Phi^+$ then $\theta \not \in {\rm sup}(\tnext \varphi)$ by maximality of $\tnext\varphi$. Therefore, $\theta \in \left(\Psi^+\setminus {\rm sup}(\tnext \varphi)\right) =\left(\remove \Psi {\tnext\varphi} \right)^+$.\medskip

\noindent\ref{ItCompTwo} If $\tnext \theta \in \Phi^-$, since $\Phi \ST \Psi$ then $\theta \in \Psi^- = \left(\remove \Psi {\tnext\varphi}\right)^-$.\medskip

\noindent\ref{ItCompThree} If $\diam \theta \in \Phi^+$, since $\Phi \ST \Psi$ then either $\theta \in \Phi^+$ or $\diam \theta \in \Psi^+$. Moreover, since $\tnext\varphi$ is maximal in $\Phi^+$ and $\diam \theta \in \Phi^+$, it follows that $\diam \theta \not \in {\rm sup}(\tnext \varphi)$. Therefore, $\diam \theta \in \left(\Psi^+\setminus {\rm sup}(\tnext \varphi)\right) =\left(\remove \Psi {\tnext\varphi} \right)^+$.\medskip

\noindent\ref{ItCompFour} If $\diam \theta \in \Phi^-$, since $\Phi \ST \Psi$ is sensible then $ \diam \theta \in \Psi^-= \left(\remove \Psi {\tnext\varphi}\right)^-$.
\endproof
	
In the next section, we will use Theorem \ref{TheoITLc} and our results on two-sided types to show that, for $\ubox$-free formulas, validity over the class of topological spaces can be reduced to validity over the class of dynamic posets.

\section{Conservativity of the $\ubox$-free fragment}\label{SecCons}

Our goal for this section is to show that the temporal logics of dynamic posets and of dynamical systems coincide with respect to $\ubox$-free formulas:

\begin{theorem}\label{TheoConservativity}
A $\ubox$-free formula $\varphi$ is satisfiable (falsifiable) over the class of dynamic posets if and only if it is satisfiable (falsifiable) over the class of dynamical systems.
\end{theorem}

We will show this by `unwinding' a quasimodel to produce a dynamic poset.

\subsection{Weak limit models}

The unwinding procedure is similar to that in \cite{FernandezITLc}. There, the points of the `limit model' obtained from a quasimodel are those infinite paths satisfying all $\diam$-formulas in their labels. However, in order to obtain a poset rather than a topological space, we will instead work with finite paths.

\begin{definition}\label{def:term-path}
If $\mathcal Q = (W,{\peq},S,\ell)$ is a quasimodel, we say that a {\em path (on $\mathcal Q$)} is a sequence $(w_i)_{i<n} \subseteq W$ such that $w_i \mathrel S w_{i+1}$ for all $i<n-1$. We define a {\em typed path (on $\mathcal Q$)} to be a sequence $((w_i,\Phi_i))_{i < n}$ such that
\begin{enumerate}[label=(\alph*)]

\item $(w_i)_{i<n}$ is a path,

\item for all $i< n$, $\Phi_i\sqsubT \ell(w_i)$, and

\item for all $i < n - 1$, $\Phi_i \ST \Phi_{i+1}$.

\end{enumerate}
We say that $((w_i,\Phi_i))_{i < n}$ is {\em proper} if ${\rm sub}(\Phi^+_{i + 1}) \subseteq {\rm sub}(\Phi^+_{i})$ for all $i<n-1$, and {\em terminal} if $\Phi^+_{n-1} = \varnothing$.
\end{definition}

Note that we allow $\Phi_i\sqsubT \ell(w_i)$ and not only $\Phi_i = \ell(w_i)$. This will allow us to use finite paths, as temporal formulas can be `forgotten' once they have been realized.
\begin{definition}

We define the {\em weak limit model $\widehat{\mathcal Q}$} of $\mathcal Q$ as follows:

\begin{enumerate}

\item Define $\widehat W$ to be the set of terminal typed paths on $\mathcal Q$ together with the empty path, which we denote $\epsilon$.

\item For $\alpha = ((w_i,\Phi_i))_{i<n}$, $\beta= ((v_i,\Psi_i))_{i<m} \in \widehat W$, define $\alpha \mathrel{\widehat \peq} \beta$ if $n\leq m$ and for all $i<n$, $w_i \peq v_i$ and $\Phi_i \peqT \Psi_i$.

\item Define $\widehat S(((w_i,\Phi_i))_{i<n}) = ((w_{i+1},\Phi_{i+1}))_{i<n-1}$; note that $\widehat S (\epsilon) = \epsilon$.

\item If $n>0$, define $\widehat \ell (((w_i,\Phi_i))_{i<n}) = \Phi_0$. Then, set
$\widehat \ell^- (\epsilon) = \bigcup_{w\in W} \ell^{-}(w)$ and
$\widehat \ell^+ (\epsilon) = \varnothing$.

\end{enumerate}
\end{definition}

The structure $\widehat {\mathcal Q}$ we have just defined is always a deterministic quasimodel. Let us first show that it is deterministic.

\begin{lemma}\label{LemIsDP}
	If $\mathcal Q = \left( W, \peq, S, \ell\right)$ is a quasimodel then $\left({\widehat W },{\widehat \peq}, {\widehat S} \right)$ is a dynamic poset.
\end{lemma}

\begin{proof}
We have to prove the following:
\medskip

\ignore{
\noindent \textbf{${\widehat S}$ is sensible}: If $\alpha = (w_i,\Phi_i)_{i<n}$, $\beta = (v_{i},\Psi_{i})_{i < m} \in {\widehat W}$ are such that $\alpha \mathrel {\widehat S} \beta$, then by definition of $\widehat S$, $\Psi_0=\Phi_1$, which by Definition \ref{def:term-path} implies that $\Phi_0 \ST \Psi_0$.
\medskip
}

\noindent\textbf{${\widehat \peq}$ is a partial order on ${\widehat W}$:} This follows easily from the fact that $\peq$ and $\peqT$ are both partial orders.
\medskip

\ignore{
\textit{Reflexivity:} Let $\alpha = (w_i,\Phi_i)_{i<n} \in {\widehat W}$. Then, since both $\peq$ and $\peqT$ are reflexive, we see that $w_i\peq w_i$ and $\Phi_i \peqT \Phi_i$ for all $i<n$, hence $\alpha \widehat \peq \alpha$.

\textit{Antisymmetry: }take $(w_i,\Phi_i)_{i<n}$, $(v_i,\Psi_i))_{i<m} \in {\widehat W}$ be such that $(w_i,\Phi_i)_{i<n} {\widehat \peq} (v_i,\Psi_i))_{i<m}$ and $(v_i,\Psi_i))_{i<m}$ ${\widehat \peq} (w_i,\Phi_i)_{i<n}$. From the definition of ${\widehat \peq}$ it follows that $m = n$, $w_i \peq v_i$, $v_i \peq w_i$, $\Phi_i \peq \Psi_i$ and $\Psi_i \peq \Phi_i$, for all $i < n$. By the definition $\peq$ we conclude that $v_i = w_i$, $\Phi_i = \Psi_i$, for all $i < n$. Hence, $(w_i,\Phi_i)_{i<n}$, $(v_i,\Psi_i))_{i<m}$.

\textit{Transitivity:} $(w_i,\Phi_i)_{i<n}$, $(v_i,\Psi_i)_{i<m}$, $(u_i,\Omega_i)_{i<s} \in {\widehat W}$ be such that $(w_i,\Phi_i)_{i<n} {\widehat \peq} (v_i,\Psi_i)_{i<m}$ and $ (v_i,\Psi_i)_{i<m}$ ${\widehat \peq}(u_i,\Omega_i)_{i<s}$. Thanks to the definition of ${\widehat \peq}$ it follows that $n \le m \le s$; for all $i < n$, $w_i\peq v_i$ and $\Phi_i \peq  \Psi_i$ and for all $i < m$, $v_i \peq u_i$ and $\Psi_i \peq \Omega_i$. Since $\peq$ is a partial order we conclude that $n\le s$ and for all $i < n$ both $w_i \peq u_i$ and $\Psi_i \peq \Omega_i$, that is, $(w_i,\Phi_i)_{i<n} {\widehat \peq} (u_i,\Omega_i)_{i<s}$.\\}

\noindent \textbf{${\widehat S}$ is a function:} This is clear since $\widehat S(\alpha)$ is defined by removing the first element of $\alpha$ if it exists, otherwise $\widehat S(\alpha) = \alpha$, and thus $\widehat S(\alpha)$ is uniquely defined for all $\alpha \in \widehat W$.
\ignore{
we proceed by contradiction. Let us take $(w_i,\Phi_i)_{i < n}$, $(v_i,\Psi_i)_{i < m}$ and $(u_i,\Omega_i)_{i < s} \in {\widehat W}$ be such that $(w_i,\Phi_i)_{i < n} {\widehat S} (v_i,\Psi_i)_{i < m}$, $(w_i,\Phi_i)_{i < n} {\widehat S} (u_i,\Omega_i)_{i < s}$ but $ (v_i,\Psi_i)_{i < m} \not= (u_i,\Omega_i)_{i < s}$. By the definition of ${\widehat S}$, $n-1 = s = m$}
\medskip

\noindent \textbf{$\widehat S$ is monotone:}
If $((w_i,\Phi_i))_{i<n} \mathrel{\widehat \peq } ((v_i,\Psi_i))_{i<m}$, then $n\leq m$ and for all $i<n$, $w_i \peq v_1$ and $\Phi_i \peqT \Psi_i$.
If $n>0$, then we also have $n-1 \leq m-1$ and for all $i < n-1$, $w_{i+1} \peq v_{i+1}$ and $\Phi_{i+1} \peqT \Psi_{i+1}$, i.e.,
\[\widehat S (\alpha) = ((w_{i+1},\Phi_{i+1}))_{i<n-1} \mathrel{\widehat \peq} ((v_{i+1},\Psi_{i+1}))_{i<m -1} = \widehat S (\beta),\]
as needed. If $n = 0$ then $\alpha = \epsilon$, so that $\widehat S(\alpha) = \epsilon$ and clearly $\epsilon \mathrel {\widehat \peq} \beta$.
\end{proof}

\subsection{Constructing terminal paths}

Next, we must show that $\widehat {\mathcal Q}$ has `enough' paths. First we show that we can iterate the forward-confluence property.

\begin{lemma}\label{LemmPath}
If $\mathcal Q$ is a quasimodel, $((w_i,\Phi_i))_{i<n}$ is a typed path in $\mathcal Q$, and $w_0 \peq v_0$, then there is a typed path $((v_i,\Psi_i))_{i<n}$ such that $w_i \peq v_i$ and $\Phi_i \peqT \Psi_i$ for all $i<n$.
\end{lemma}

\proof
First we find $v_i$ by induction on $i$; $v_0$ is already given, and once we have found $v_i$, we use forward confluence to choose $v_{i+1}$ so that $v_i \mathrel S v_{i+1}$ and $w_{i+1} \peq v_{i+1}$. Then we set $\Psi_i = \ell(v_i)$; since $S$ is sensible, $\Psi_{i} \ST \Psi_{i+1}$, and by Lemma \ref{LemmRestrict}.\ref{ItCrossTrans}, $\Phi_n \peqT \Psi_n$.
\endproof

We want to prove that any point can be included in a terminal typed path. For this we will first show that we can work mostly with properly typed paths, thanks to the following.

\begin{lemma}\label{LemProperPath}
Let $\mathcal Q= (W,{\peq},S,\ell)$ be a quasimodel, $(w_i)_{i <n}$ be a path on $W$, and $\Phi_0 \sqsubseteq \ell (w_0)$. Then there exist $(\Phi_i)_{i<n}$ such that $((w_i,\Phi_i))_{i<n}$ is a properly typed path.
\end{lemma}

\proof
For $i < n-1$ define recursively $\Phi_{i+1} = \ell (w_{i + 1}) \upharpoonright {\rm sub} ( \Phi^+_i ) $; by the assumption that $S$ is sensible and Lemma \ref{LemmRestrict}, $(\Phi_i,\Phi_{i+1})$ is sensible for each $i<n-1$. It is easy to see that $((w_i,\Phi_i))_{i<n}$ thus defined is proper.
\endproof

However, the properly typed paths we have constructed need not be terminal. This will typically require extending them to a long-enough path, as we do below.

\begin{lemma}\label{LemTerminal}
If $\mathcal Q$ is a quasimodel, then any non-empty typed path on $\mathcal Q$ can be extended to a terminal path.
\end{lemma}

\proof
Let $\mathcal Q= (W,{\peq},S,\ell)$ and $\alpha = ((w_i,\Phi_i))_{i<m}$ be any typed path on $\mathcal Q$. For a type $\Phi$, define $\|\Phi\| = |{\rm sub} (\Phi^+)|$.
We proceed to prove the claim by induction on $\|\Phi_{m-1}\|$. Consider first the case where $\Phi_{m-1}^+$ contains no temporal formulas; that is, formulas of the form $\tnext \psi$ or $\diam \psi$ for some $\psi$. In this case, using the seriality of $S$ choose $w_{m}$ such that $w_{m-1} \mathrel S w_{m}$, and define $\Phi_{m+1} = (\ell^-(w_{m});\varnothing)$; it is easy to see that $((w_i,\Phi_i))_{i\leq m}$ is a terminal path.

Otherwise, let $\varphi$ be a maximal temporal formula of $\Phi^+_{m-1}$, i.e., it does not appear as a proper subformula of any other temporal formula in $\Phi^+_{m-1}$.
We consider two sub-cases.

Assume first that $\varphi = \tnext \psi$. Then, by the seriality of $S$, we may choose $w_m$ so that $w_{m-1} \mathrel S w_{m}$.
Applying Lemma \ref{LemProperPath}, let $\widetilde \Phi_m$ be such that $((w_{m-1},\Phi_{m-1}), (w_m,\widetilde \Phi_m))$ is a properly typed path.
Setting $\Phi_m = \remove {\widetilde \Phi_m }{\tnext\psi}$, we see by Lemma \ref{LemDelete}.\ref{ItDeleteCirc} that
\[((w_{m-1},\Phi_{m-1}),(w_m,\Phi_m))\]
is a properly typed path, and $\| \Phi_m \| < \|  \Phi_{m-1} \|$, since the left-hand side does not count $\tnext\psi$. Thus we may apply the induction hypothesis to obtain a terminal typed path $((w_i,\Phi_i))_{i<n}$ extending $\alpha$.

Now consider the case where $\varphi = \diam \psi$. Since $S$ is $\omega$-sensible, there is a path
\[w_{m-1} \mathrel S w_m \mathrel S \ldots \mathrel S w_k\]
so that $\varphi \in \ell( w_k )$.
Using the seriality of $S$, choose $w_{k+1}$ so that $w_k \mathrel S w_{k+1}$.

By Lemma \ref{LemProperPath}, there are types $\Phi_i$ for $m\leq i \leq k$ and a type $\widetilde \Phi_{k+1}$ such that
\[((w_{m-1},\Phi_{m-1}), \ldots, (w_k,\Phi_k), (w_{k+1},\widetilde \Phi_{k+1}))\]
is a properly typed path. Then, define $\Phi_{k + 1} = \remove{\widetilde \Phi_{k+1} }{\diam \psi}$. Using Lemma \ref{LemDelete} we see that $(\Phi_k,\Phi_{k+1})$ is sensible; moreover, $\|  \Phi_{k+1} \| < \| \Phi_{m-1} \|$. Hence we can apply the induction hypothesis to obtain a terminal typed path $((w_i,\Phi_i))_{i <n}$ extending $\alpha$.
\endproof

\subsection{From weak limit models to models}

With this, we are ready to show that our unwinding is indeed a deterministic quasimodel.

\begin{lemma}\label{LemIsQuasi}
If $\mathcal Q$ is a quasimodel, then $\widehat{\mathcal Q}$ is a deterministic quasimodel.
\end{lemma}

\proof
Let $\mathcal Q = (W,{\peq}, S, \ell)$. We have already seen in Lemma \ref{LemIsDP} that $(\widehat W,\widehat \peq, \widehat S )$ is a dynamic poset, so it remains to check that $(\widehat W,\widehat \peq, \widehat \ell)$ is a labelled frame and $\widehat S$ is sensible and $\omega$-sensible. Let $\alpha = \big ( (w_i,\Phi_i) \big ) _{i<n} \in \widehat W$.

First we must check that if $\alpha \mathrel{\widehat \peq} \beta$, then $\widehat \ell(\alpha) \peqT \widehat \ell(\beta)$. Consider two cases; if $n>0$, then $\beta$ is also of the form $(v_i,\Psi_i)_{i<m}$ with $m>0$ and by definition,
$\widehat \ell(\alpha) = \Phi_0 \peqT \Psi_0 = \widehat \ell( \beta ).$
Otherwise, $\alpha = \epsilon$, and it is clear from the definition of $\widehat \ell(\epsilon )$ that $\widehat \ell(\epsilon ) \peqT \widehat \ell( \beta )$ regardless of $\beta$.

Now assume that $\varphi \to \psi \in \widehat \ell^-(\alpha)$. If $n>0$, then since $\mathcal Q$ is a labeled frame, we can pick $v_0 \seq w_0$ with $ \varphi \in \ell^+(v_0)$ and $\psi \in \ell^-(v_0)$. Since $\Phi_0 \sqsubT \ell(w_0) \peqT \ell(v_0)$, by Lemma \ref{LemmRestrict}.\ref{ItCrossTrans}, $\Phi_0 \peqT \ell(v_0)$, so that by Lemma \ref{LemmPath}, there is a typed path $\beta' = ((v_i,\Psi_i))_{i<n}$ with $\Psi_0 = \ell(v_0)$ such that $w_i \peq v_i$ and $\Phi_i \peqT \Psi_i$ for all $i<n$. By Lemma \ref{LemTerminal}, we can extend $\beta'$ to a terminal path $\beta$. Then, it is easy to see that $\alpha \mathrel{\widehat\peq} \beta$, $ \varphi \in \widehat \ell^+( \beta )$, and $\psi \in \widehat \ell^-( \beta )$, as required.

To check that every pair in $\widehat S$ is sensible, consider two cases. If $\widehat S(\alpha) \not = \epsilon$, then $\alpha$ has length at least two, but since $\alpha$ is a typed path,
\[\widehat \ell(\alpha) = \Phi_0 \ST \Phi_1 = \widehat \ell \big ( \widehat S (\alpha) \big ) .\]
Otherwise, $\widehat S(\alpha) = \epsilon$; this means that either $\alpha = \epsilon$ and thus $\widehat \ell ^+ (\alpha) = \varnothing$, or $\alpha$ has length $1$, in which case since $\alpha$ is terminal, so we also have that $\widehat \ell ^+ (\alpha) = \varnothing$. In either case, there can be no temporal formula in $\widehat \ell ^+ (\alpha)$. Now, if $\tnext\varphi \in \widehat \ell ^- (\alpha)$, then $\tnext \varphi \in \ell ^- (w )$ for some $w \in W$, hence $\varphi \in \ell ^- S(v)$ for any $v$ with $w \mathrel S v$ (which exists since $S$ is serial), and thus $\varphi \in \widehat \ell^- (\epsilon)$.
Similarly, if $\diam\varphi \in \widehat \ell ^- (\alpha)$, then by Definition \ref{def:type}.\ref{cond:type:diam} $ \varphi \in  \ell ^- ( \alpha )$, so that $ \varphi \in \widehat \ell ^- (\epsilon)$.

Finally we check that $\widehat S$ is $\omega$-sensible. Suppose that $\diam \varphi \in \ell^+ (\alpha)$. This means that $\alpha \not = \epsilon$, so $\alpha$ is terminal, and hence $n>0$ and $\diam \alpha \not \in \Phi_{n-1}$. But this is only possible if $\varphi \in \Phi_i$ for some $i<n-1$, in which case $\varphi \in \widehat \ell \big ( \widehat S ^i (\alpha) \big ) $.
\endproof

Let us put all of our work together to prove our main result.

\proof[Proof of Theorem \ref{TheoConservativity}]
Suppose that $\varphi \in \mathcal L_\diam$ is satisfied (falsified) on a dynamical topological model. Then, by Theorem \ref{TheoITLc}, $\varphi$ is satisfied (falsified) on some point $w_\ast$ of a ${\rm sub}(\varphi)$-quasimodel $\mathcal Q = (W,{\peq},S, \ell )$. By Lemma \ref{LemIsQuasi}, $\widehat {\mathcal Q}$ is a deterministic quasimodel, and by Lemma \ref{LemTerminal}, $(w_\ast, \ell(w_\ast))$ can be extended to a terminal path $\alpha_\ast \in \widehat W$. By Lemma \ref{LemTruth}, $\alpha_\ast $ satisfies (falsifies) $\varphi$ on the dynamic poset model $(\widehat W,\widehat \peq, \widehat S, \val\cdot^{\widehat \ell})$.
\endproof

\section{Concluding Remarks}\label{SecConc}

We have proposed a natural `minimalist' intuitionistic temporal logic, $\logbasic$, along with possible extensions including Fischer Servi or constant domain axioms. We have seen that relational semantics validate the constant domain axiom, leading us to consider a wider class of models based on topological spaces, with a novel interpretation for `henceforth' based on invariant neighbourhoods. With this, we have shown that the logics $\logbasic$, $\logexp$, $\loghomeo$ and $\logpers$ are sound for the class of all dynamical systems, of all dynamical posets, of all open dynamical systems, and of all persistent dynamical posets, respectively, which we have used in order to prove that the logics are pairwise distinct.

Of course this immediately raises the question of completeness, which we have not addressed. Specifically, the following are left open.

\begin{question}
Are $\logbasic$ and $\logbasic _\ubox $ complete for the class of dynamical systems?
\end{question}

\begin{question}
Are $\logexp$, $\logbasic_\diam$ and $\logexp_\ubox$ complete for the class of dynamic posets?
\end{question}

\begin{question}
Are $\loghomeo$, $\loghomeo_\diam$ and $\loghomeo_\ubox$ complete for the class of open dynamical systems?
\end{question}

\begin{question}
Are $\logpers$, $\logpers_\diam$ and $\logpers_\ubox$ complete for the class of persistent dynamic posets?
\end{question}

Note that by Theorem \ref{TheoConservativity}, $\logbasic_\diam$ is complete for the class of dynamical systems if and only if it is complete for the class of dynamic posets, so thanks to the results we have shown here, proving topological completeness would give us Kripke completeness for free. It is likely that the techniques employed in \cite{dtlaxiom}, also based on non-deterministic quasimodels, could be adapted to the intuitionistic setting to obtain such a result. The completeness of $\loghomeo$ and $\logpers$ is likely to be a more difficult problem, as in these cases it is not even known if the set of valid formulas is computably enumerable.

\begin{question}
Are the sets of formulas of $\mathcal L$, $\mathcal L_\diam$, or $\mathcal L_\ubox$ valid over the class of all persistent dynamic posets computably enumerable?
\end{question}

\begin{question}
Are the sets of formulas of $\mathcal L$, $\mathcal L_\diam$, or $\mathcal L_\ubox$ valid over the class of all open dynamical systems computably enumerable?
\end{question}

In both cases a negative answer is possible, since that is the case for their classical counterparts \cite{wolter}.
Nevertheless, the proofs of non-axiomatizability in the classical case do not carry over to the intuitionistic setting in an obvious way, and these remain challenging open problems.
%


\bibliographystyle{plain}

\begin{thebibliography}{10}

\bibitem{alek}
Pavel Aleksandroff.
\newblock Diskrete r\"aume.
\newblock {\em Matematicheskii Sbornik}, 2:501--518, 1937.

\bibitem{arte}
Sergei~N. Art\"emov, Jennifer~M. Davoren, and Anil Nerode.
\newblock Modal logics and topological semantics for hybrid systems.
\newblock {\em Technical Report MSI 97-05}, 1997.

\bibitem{IMLA}
Philippe Balbiani, Joseph Boudou, Mart\'{\i}n Di\'eguez, and David
  Fern\'andez-Duque.
\newblock Bisimulations for intuitionistic temporal logics.
\newblock In {\em Intuitionistic Modal Logic and Applications {(IMLA)}}, 2017.
\newblock Forthcoming.

\bibitem{BalbianiDieguezJelia}
Philippe Balbiani and Mart\'{i}n Di\'eguez.
\newblock Temporal here and there.
\newblock In M.~Loizos and A.~Kakas, editors, {\em Logics in Artificial
  Intelligence}, pages 81--96. Springer, 2016.

\bibitem{BoudouCSL}
Joseph Boudou, Mart{\'{\i}}n Di{\'{e}}guez, and David Fern{\'{a}}ndez{-}Duque.
\newblock A decidable intuitionistic temporal logic.
\newblock In {\em 26th {EACSL} Annual Conference on Computer Science Logic
  ({CSL})}, volume~82, pages 14:1--14:17, 2017.

\bibitem{Brewka11}
Gerhard Brewka, Thomas Eiter, and Miros\l{}aw Truszczy\'{n}ski.
\newblock Answer set programming at a glance.
\newblock {\em Communications of the ACM}, 54(12):92--103, 2011.

\bibitem{CabalarD14}
Pedro Cabalar and Mart{\'{\i}}n Di{\'{e}}guez.
\newblock Strong equivalence of non-monotonic temporal theories.
\newblock In {\em 14th International Conference on Principles of Knowledge
  Representation and Reasoning ({KR'14})}, 2014.

\bibitem{CP07}
Pedro Cabalar and Gilberto P{\'e}rez~Vega.
\newblock Temporal equilibrium logic: A first approach.
\newblock In {\em Computer Aided Systems Theory -- EUROCAST'07}, pages
  241--248. Springer Berlin Heidelberg, 2007.

\bibitem{Davies96}
Rowan Davies.
\newblock A temporal-logic approach to binding-time analysis.
\newblock In {\em Proceedings, 11th Annual {IEEE} Symposium on Logic in
  Computer Science, New Brunswick, New Jersey, USA, July 27-30, 1996}, pages
  184--195, 1996.

\bibitem{Davoren2009}
Jennifer~M. Davoren.
\newblock {On intuitionistic modal and tense logics and their classical
  companion logics: Topological semantics and bisimulations}.
\newblock {\em Annals of Pure and Applied Logic}, 161(3):349--367, 2009.

\bibitem{DavorenIntuitionistic}
Jennifer~M. Davoren, Vaughan Coulthard, Thomas Moor, Rajeev Goré, and Anil
  Nerode.
\newblock Topological semantics for intuitionistic modal logics, and spatial
  discretisation by {A}/{D} maps.
\newblock In {\em Workshop on Intuitionistic Modal Logic and Applications
  (IMLA)}, 2002.

\bibitem{DeGroote1995}
Philippe {De Groote}.
\newblock {\em The Curry-Howard Isomorphism}, volume~8.
\newblock Cahiers du Centre de logique, Academia, 1995.

\bibitem{JH03}
Dick De~Jongh and Lex Hendriks.
\newblock Characterization of strongly equivalent logic programs in
  intermediate logics.
\newblock {\em Theory and Practice of Logic Programming}, 3(3):259--270, 2003.

\bibitem{CerroHS15}
Luis {Fari{\~{n}}as del Cerro}, Andreas Herzig, and Ezgi {Iraz Su}.
\newblock Epistemic equilibrium logic.
\newblock In {\em Proceedings of the 24th International Joint Conference on
  Artificial Intelligence, {IJCAI}}, pages 2964--2970, 2015.

\bibitem{FernandezITLc}
D.~{Fern{\'a}ndez-Duque}.
\newblock The intuitionistic temporal logic of dynamical systems.
\newblock {\em ArXiv e-prints}, 2016.

\bibitem{dtlaxiom}
David Fern\'andez{-}Duque.
\newblock A sound and complete axiomatization for dynamic topological logic.
\newblock {\em Journal of Symbolic Logic}, 77(3):947--969, 2012.

\bibitem{FS84}
Gis\`ele Fischer~Servi.
\newblock {Axiomatisations for some intuitionistic modal logics}.
\newblock In {\em {Rendiconti del Seminario Matematico}}, volume~42, pages
  179--194. Universitie Politecnico Torino, 1984.

\bibitem{gekakasc14b}
Martin Gebser, Roland Kaminski, Benjamin Kaufmann, and Torsten Schaub.
\newblock Clingo= asp+ control: Preliminary report.
\newblock {\em arXiv preprint arXiv:1405.3694}, 2014.

\bibitem{gekasc11c}
Martin Gebser, Roland Kaminski, and Torsten Schaub.
\newblock aspcud: A {L}inux package configuration tool based on answer set
  programming.
\newblock In {\em Proceedings of the Second International Workshop on Logics
  for Component Configuration (LoCoCo'11)}, volume~65, pages 12--25, 2011.

\bibitem{GebserSTV11}
Martin Gebser, Torsten Schaub, Sven Thiele, and Philippe Veber.
\newblock Detecting inconsistencies in large biological networks with answer
  set programming.
\newblock {\em {Theory and Practice on Logic Programming}}, 11(2-3):323--360,
  2011.

\bibitem{Hey30}
Arend Heyting.
\newblock {\em Die formalen Regeln der intuitionistischen Logik}.
\newblock Sitzungsberichte der Preussischen Akademie der Wissenschaften.
  Physikalisch-mathematische Klasse. De{\"u}tsche Akademie der Wissenschaften
  zu Berlin, Mathematisch-Naturwissenschaftliche Klasse, 1930.

\bibitem{KamideBounded}
Norihiro Kamide and Heinrich Wansing.
\newblock Combining linear-time temporal logic with constructiveness and
  paraconsistency.
\newblock {\em Journal of Applied Logic}, 8(1):33--61, 2010.

\bibitem{KojimaNext}
Kensuke Kojima and Atsushi Igarashi.
\newblock Constructive linear-time temporal logic: Proof systems and {K}ripke
  semantics.
\newblock {\em Information and Computation}, 209(12):1491 --1503, 2011.

\bibitem{konev}
Boris Konev, Roman Kontchakov, Frank Wolter, and Michael Zakharyaschev.
\newblock Dynamic topological logics over spaces with continuous functions.
\newblock In G.~Governatori, I.~Hodkinson, and Y.~Venema, editors, {\em
  Advances in Modal Logic}, volume~6, pages 299--318, London, 2006. College
  Publications.

\bibitem{wolter}
Boris Konev, Roman Kontchakov, Frank Wolter, and Michael Zakharyaschev.
\newblock On dynamic topological and metric logics.
\newblock {\em Studia Logica}, 84:129--160, 2006.

\bibitem{KremerIntuitionistic}
Philip Kremer.
\newblock A small counterexample in intuitionistic dynamic topological logic.
\newblock \url{http://individual.utoronto.ca/
  philipkremer/onlinepapers/counterex.pdf}, 2004.
\newblock [Online; accessed \today].

\bibitem{kmints}
Philip Kremer and Grigori Mints.
\newblock Dynamic topological logic.
\newblock {\em Annals of Pure and Applied Logic}, 131:133--158, 2005.

\bibitem{LPF+06}
Nicola Leone, Gerald Pfeifer, Wolfgang Faber, Thomas Eiter, Georg Gottlob,
  Simona Perri, and Francesco Scarcello.
\newblock {The DLV system for knowledge representation and reasoning}.
\newblock {\em ACM Transactions on Computational Logic}, 7:499--562, 2006.

\bibitem{temporal}
Orna Lichtenstein and Amir Pnueli.
\newblock Propositional temporal logics: Decidability and completeness.
\newblock {\em Logic Jounal of the IGPL}, 8(1):55--85, 2000.

\bibitem{LPV01}
Vladimir Lifschitz, David Pearce, and Agust\'{\i}n Valverde.
\newblock Strongly equivalent logic programs.
\newblock {\em ACM Transactions in Computational Logic}, 2(4):526--541, 2001.

\bibitem{Maier2004Heyting}
Patrick Maier.
\newblock Intuitionistic {LTL} and a new characterization of safety and
  liveness.
\newblock In Jerzy Marcinkowski and Andrzej Tarlecki, editors, {\em 18th
  {EACSL} Annual Conference on Computer Science Logic ({CSL})}, pages 295--309,
  Berlin, Heidelberg, 2004. Springer Berlin Heidelberg.

\bibitem{MintsInt}
Grigori Mints.
\newblock {\em A Short Introduction to Intuitionistic Logic}.
\newblock University Series in Mathematics. Springer, 2000.

\bibitem{munkres2000}
James~R. Munkres.
\newblock {\em Topology}.
\newblock Featured Titles for Topology Series. Prentice Hall, Incorporated,
  2000.

\bibitem{PacuitNeighborhood}
Eric Pacuit.
\newblock {\em Neighborhood Semantics for Modal Logic}.
\newblock Springer, 2017.

\bibitem{Pearce96}
David Pearce.
\newblock A new logical characterisation of stable models and answer sets.
\newblock In J{\"u}rgen Dix, Lu{\'{\i}}s~Moniz Pereira, and Teodor~C.
  Przymusinski, editors, {\em Non-Monotonic Extensions of Logic Programming},
  pages 57--70. Springer Berlin Heidelberg, 1997.

\bibitem{Simpson94}
Alex~K. Simpson.
\newblock {\em The proof theory and semantics of intuitionistic modal logic}.
\newblock PhD thesis, University of Edinburgh, {UK}, 1994.

\bibitem{tarski}
Alfred Tarski.
\newblock Der {A}ussagenkalk\"ul und die {T}opologie.
\newblock {\em Fundamenta Mathematica}, 31:103--134, 1938.

\bibitem{WBS15}
Przemys{\l}aw~Andrzej Wa{\l}{\k{e}}ga, Mehul Bhatt, and Carl Schultz.
\newblock Aspmt(qs): Non-monotonic spatial reasoning with answer set
  programming modulo theories.
\newblock In {\em Logic Programming and Nonmonotonic Reasoning 2017}, pages
  488--501. Springer, 2015.

\end{thebibliography}


\end{document}